\documentclass[letterpaper,10pt,reqno,onefignum,onetabnum]{amsart}
\usepackage[english]{babel}
\usepackage{amsmath}
\usepackage{amsthm}
\usepackage[foot]{amsaddr}
\usepackage[dvipsnames]{xcolor}
\usepackage{hyperref}
\hypersetup{
	colorlinks = true,
	linkcolor = OliveGreen,
	anchorcolor = OliveGreen,
	citecolor = OliveGreen,
	filecolor = OliveGreen,
	urlcolor = OliveGreen
}
\usepackage{algorithm}% http://ctan.org/pkg/algorithms
\usepackage{algorithmic}% http://ctan.org/pkg/algorithms
\usepackage{float}
\usepackage{lipsum}
\usepackage{amsfonts}
\usepackage{amssymb}
\usepackage{graphicx}
\usepackage{epstopdf}
\usepackage{cases}
\usepackage{multirow}
\ifpdf
\DeclareGraphicsExtensions{.eps,.pdf,.png,.jpg}
\else
\DeclareGraphicsExtensions{.eps}
\fi

\newtheorem{teo}{Theorem}[section]     % numbered within section
\newtheorem{prop}[teo]{Proposition}    % shares teo counter
\newtheorem{lem}[teo]{Lemma}

\newtheorem{pro}[teo]{Problem}
\newtheorem{algo}[teo]{Algorithm}
\newtheorem{asume}[teo]{Assumption}

\newtheorem{rem}[teo]{Remark}

\newtheorem{defi}[teo]{Definition}

\usepackage{tikz}
\usepackage{booktabs}%
\usepackage{hhline}
\usetikzlibrary{matrix}
\usetikzlibrary{arrows}
\usepackage{verbatim}
\usepackage{mathrsfs}
\usepackage{bm}
\usepackage{color}
\usepackage{xcolor}
\usepackage{caption}
\usepackage{subcaption}
\usepackage{enumitem}
\newcommand{\N}{\mathbb N}

\newcommand{\R}{\mathbb R}

\renewcommand{\H}{\mathcal{H}}
\newcommand{\G}{\mathcal G}

\newcommand{\HH}{{\bm{\mathcal{H}}}}

\newcommand{\Id}{{\bf Id}}
\newcommand{\id}{\textnormal{Id}}
\newcommand{\x}{\bm x}

\newcommand{\w}{\bm w}

\newcommand{\T}{\tau}
\newcommand{\weak}{\rightharpoonup}
\newcommand{\ran}{\textnormal{ran}\,}
\newcommand{\dom}{\textnormal{dom}\,}
\newcommand{\epi}{\textnormal{epi}\,}

\newcommand{\zer}{\textnormal{zer}}

\newcommand{\gra}{\textnormal{gra}\,}

\newcommand{\scal}[2]{{\left\langle{{#1}\mid{#2}}\right\rangle}}

\newcommand{\menge}[2]{\big\{{#1}~\big |~{#2}\big\}}

\newcommand{\RPP}{\ensuremath{\left]0,+\infty\right[}}

\newcommand{\sri}{\ensuremath{\text{\rm sri}\,}}

\newcommand{\prox}{\ensuremath{\text{\rm prox}\,}}

\usepackage{geometry}
\numberwithin{equation}{section}

\numberwithin{equation}{section}

\floatname{algorithm}{Line Search}

\DeclareFontEncoding{FMS}{}{}
\DeclareFontSubstitution{FMS}{futm}{m}{n}
\DeclareFontEncoding{FMX}{}{}
\DeclareFontSubstitution{FMX}{futm}{m}{n}
\DeclareSymbolFont{fouriersymbols}{FMS}{futm}{m}{n}
\DeclareSymbolFont{fourierlargesymbols}{FMX}{futm}{m}{n}
\DeclareMathDelimiter{\nr}{\mathord}{fouriersymbols}{152}{fourierlargesymbols}{147}

\DeclareMathDelimiter{\nr}{\mathord}{fouriersymbols}{152}{fourierlargesymbols}{147}
\DeclareMathAlphabet{\mathpzc}{OT1}{pzc}{m}{it}

\title[Inexact Warped Resolvent iterations]{\vspace*{-2cm}Inexact Warped Resolvent iterations}
\author{Raul T. Marcavillaca$^1$}
\address{$^1$ Departamento de Ingeniería Matemática, Universidad de Concepción, Concepción, Chile. {\it 
		E-mail address:} 
	{\sf{raultm.rt@gmail.com;
rtintaya@udec.cl}}.}
\author{Fernando Rold\'an$^2$}
\address{$^2$Departamento de Ingeniería Matemática and CI$^2$MA, Universidad de Concepción, Concepción, Chile. {\it 
		E-mail address:} 
	{\sf{fernandoroldan@udec.cl}}. }

\begin{document}
	\begin{abstract}
In this paper we aim to solve structured monotone inclusions using inexact warped resolvents evaluated under a relative-error criterion.  The resulting algorithms admit a geometric interpretation as relaxed projection methods onto dynamically generated cuts, extending classical projection--proximal and hybrid extragradient proximal frameworks to nonlinear warped resolvent. Under mild assumptions, we establish weak convergence of the iterates. We further derive strong convergence results by incorporating projection steps onto intersections of halfspaces via Haugazeau-type scheme, as well as linear convergence under a metric subregularity assumption. The proposed algorithm provides a unified framework for incorporating inexact resolvent computations into several classical schemes arising in monotone operator theory and primal-dual optimization, such as Tseng's forward-backward-forward splitting, forward-backward-half-forward, Chambolle--Pock, and Condat--V\~u. Finally, we present applications in saddle-point and structured convex minimization problems. Numerical experiments on synthetic saddle-point instances and computed tomography reconstruction demonstrate the computational advantages of the proposed methods.
	
		\par
		\bigskip
		
		%62H35, 94A08,
	\end{abstract}
	
	\maketitle
\section{Introduction}
Monotone inclusion problems constitute a broad and flexible framework encompassing several models arising in optimization \cite{Combettes2018MP}, variational inequalities \cite{bauschkebook2017}, equilibrium problems \cite{Combettes2005equilibrium}, partial differential equations \cite{AubinHelene2009,Glowinsky1975,Showalter1997}, signal processing and imaging \cite{BotHendrich2014TV,Briceno2011ImRe,BurgerSawatzkySteidl2014,chambolle2016AN}, traffic theory \cite{Nets1,GafniBert84}, machine learning \cite{BotrelaxFBF2023,Nocedal2018,CombettesPesquet2021strategies}, among others. In this work, we consider the following structured monotone inclusion problem.
\begin{pro}\label{pro:main}
Let $\HH$ be a real Hilbert space, let $\bm{A} \colon \HH \to 2^{\HH}$ be a maximally monotone operator, and let $\bm{C} \colon \HH \to \HH$ be a cocoercive operator. The problem is to
\begin{equation}\label{eq:promain}
\text{find } \x \in \HH \text{ such that } 0 \in \bm{A}\x+\bm{C}\x,
\end{equation}
under the assumption that the solution set, denoted by $\mathcal S$, is nonempty.
\end{pro}
Problem~\ref{pro:main} covers a large class of structured optimization models, variational systems and inclusion problems with complex structures. 
For instance, given real Hilbert spaces $\H$ and $\G$, the problem 
\begin{equation}\label{eq:probPDintro}
 \textnormal{find } x \in \H \textnormal{ such that } 0\in Ax+L^\ast BLx+Cx+Dx,
\end{equation}
where $A\colon \H \to 2^\H$ and $B\colon \G \to 2^\G$ are maximally monotone operators, $L \colon \H \to \G$ is a bounded linear operator, $C\colon \H \to \H$ is a cocoercive operator, and $D\colon \H \to \H$ is a Lipschitz continuous operator, can be reformulated as an instance of Problem~\ref{pro:main} through suitable primal-dual embeddings in the space $\HH = \H \times \G$ (see for instance \cite{MorinBanertGiselsson2022}).

Operator splitting methods for solving \eqref{eq:promain} have been extensively studied in the literature. For instance, one of the most popular algorithms is the forward-backward (FB) splitting \cite{Chenhg1997,passty1979JMAA}, which, for $x_0 \in \HH$ and $\gamma \in ]0,2\beta[$, where $\beta \in \RPP$ is the cocoercivity constant of $\bm{C}$, iterates as follows:
    \begin{equation}\label{eq:algFB}
	(\forall n\in\N)\quad x_{n+1} = J_{\gamma \bm{A}}(x_n-\gamma \bm{C}x_n).
\end{equation}
The sequence $(x_n)_{n \in \N}$ generated by FB converges weakly to a point in $\mathcal{S}$. In recent years, variants of FB including variable metrics and {\it warped resolvents} (also called nonlinear-forward-backward algorithms) have attracted considerable attention due to their flexibility and their ability to exploit the structure of large-scale optimization problems \cite{BrediesSun2015,BrediesSun17,Bredies2022DPP,BuiCombettesWarped2020,Giselsson2021NFBS,MorinBanertGiselsson2022,Xue2023}. Given a single valued operator $\bm M:\HH\to\HH$, the warped resolvent is defined by $J_{\bm{A}}^{\bm{M}} = (\bm{M}+\bm{A})^{-1}\circ \bm{M}$ and it was introduced in \cite{BuiCombettesWarped2020,Giselsson2021NFBS}. Note that $J_{\bm{A}}^{\Id}$ is the standard resolvent.  When $\bm M=\nabla f$, where $f$ is a differentiable convex function, one obtains Bregman-type resolvent operators \cite{bui2021bregman} .
For instance, given a nonnegative sequence $(\lambda_n)_{n \in \N}$, an algorithm generated by warped resolvents, introduced in \cite[Theorem~4.12]{combettes2024geometry},  iterates as follows
\begin{equation}\label{eq:algCombettes}
	(\forall n\in\N)\quad 
	\left\lfloor
	\begin{array}{l}
		w_n = J_{\bm{A}+\bm{C}}^{\bm{M}}x_n\\
		w_n^* = \bm{M}x_n-\bm{M}w_n-\bm{C}w_n\\
		t_n^* = w_n^*+\bm{C}x_n\\
		\delta_n = \scal{x_n-w_n}{t_n^*}-\frac{1}{4\beta}\|w_n-x_n\|^2 \\
		d_n = \begin{cases}
			\dfrac{\delta_n}{\|t^*_n\|^2} t^*_n, &\textnormal{ if } \delta_n >0; \\
			0,  &\textnormal{ otherwise } 
		\end{cases}\\
		x_{n+1}=x_n-\lambda_n d_n.
	\end{array}
	\right.
\end{equation}
 Under mild assumptions on $\bm{M}$ and $(\lambda_n)_{n \in \N}$, $(x_n)_{n \in \N}$ also converges weakly to a point in $\mathcal{S}$. A long list of popular splitting methods can be recovered from warped resolvent iterations, such as FB, forward-backward-forward (FBF) \cite{BuiCombettesWarped2020,Tseng2000SIAM}, forward-backward-half-forward (FBHF) \cite{BricenoDavis2018,Giselsson2021NFBS}, Chambolle--Pock \cite{ChambollePock2011,MorinBanertGiselsson2022}, Condat--V\~u \cite{Condat13,MorinBanertGiselsson2022,Vu13}, forward-reflected-backward (FRB) \cite{Malitsky2020SIAMJO,MorinBanertGiselsson2022}, Douglas--Rachford~\cite{Bredies2022DPP,Lions1979SIAM}, and forward-primal-dual-half-forward (FPDHF) \cite{MaulenRoldanVega2026,roldan2025forward}.

From a practical viewpoint, exact evaluations of proximal or resolvent operators are often computationally expensive or even impossible. Consequently, inexact proximal frameworks and relative-error criteria have become fundamental tools in modern optimization and monotone operator theory; see, for instance, \cite{Alves2020,AlvesLorenzNaldi2026,Eckstein2017MP,Solodov1999SVA}. In particular, the hybrid projection-proximal point method and the hybrid extragradient-proximal point method of Solodov and Svaiter \cite{Solodov1999SVA,solodov2000forcing} use approximate proximal information to construct separating hyperplanes or extragradient directions. Recently, an inexact version of the degenerate proximal point algorithm \cite{Bredies2022DPP} was proposed in \cite{AlvesLorenzNaldi2026} and, as a consequence, inexact versions of Chambolle--Pock and the Davis--Yin method \cite{DavisYin2017} were derived. For example, given $x \in \HH$, $p=J_{\bm{A}} x $ can be approximated by solving the system 
\begin{align*}
\begin{cases}
v \in \bm{A}p\\
e=v + x - p ,\\
\end{cases}
\end{align*}
where $e \in \HH$ is the error of the approximation. Note that, if $e = 0$, then $v = x-p$ and $p = J_{\bm{A}} x$.

In this paper, we introduce a relative-error warped resolvent framework based on geometrically generated separating halfspaces. The proposed methods admit a natural interpretation as relaxed projection schemes onto dynamically generated cuts, extending classical projection-proximal and hybrid extragradient proximal methodologies \cite{Solodov1999SVA,Solodov1999JCA,Solodov2001} to the nonlinear warped resolvent setting.  Our setting incorporates variable metrics through two positive definite operators $\bm P$ and $\bm S$, thereby providing substantial flexibility and allowing us to encompass several important algorithms from the literature within a unified geometric perspective. Our convergence analysis first establishes Fej\'er monotonicity, summability properties, and weak convergence of the generated iterates under mild assumptions. Moreover, by an adequate choice of the relaxation sequence $(\lambda_n)_{n \in \N}$, we also provide a warped resolvent algorithm with explicit step-sizes which has a simpler structure, which facilitates its numerical implementation.  We further obtain strong convergence by incorporating projection steps onto intersections of separating halfspaces, in the spirit of Haugazeau-type methods. In addition, linear convergence is derived under a metric subregularity assumption. As a consequence, we derive inexact versions of several well-known splitting schemes including FB, FBF, FBHF, Chambolle--Pock, Condat--V\~u, and FPDHF. These methods allow for inexact evaluations of the backward steps under relative-error criteria. We provide applications in saddle-point problems and structured convex minimization. In addition, numerical implementations on large-scale synthetic saddle-point instances and computed tomography reconstruction are presented to illustrate the numerical advantages and practical performance of the proposed algorithms.

The remainder of the paper is organized as follows. Section~\ref{sec:pre} collects notation and preliminary material. Section~\ref{sec:main} introduces the inexact warped resolvent framework and proves its basic convergence properties. In addition, we establish strong and linear convergence results. Section~\ref{sec:particular} derives several inexact splitting schemes as particular instances of the proposed framework. Section~\ref{sec:applications} presents applications and numerical experiments. Finally, Section~\ref{sec:conclu} provides concluding remarks.

\section{Preliminaries}\label{sec:pre}
	
	In this section we collect notation, definitions, and basic auxiliary results used throughout the paper. Throughout the paper, $\mathcal{H}$ and $\mathcal{G}$ denote real Hilbert spaces endowed with inner product $\scal{\cdot}{\cdot}$ and induced norm $\|\cdot\|$, and $2^{\mathcal{H}}$ denotes the power set of $\mathcal{H}$. We denote by $\mathcal{P}(\mathcal{H})$ the set of bounded linear self-adjoint operators $P\colon\H\to\H$ which are positive definite, in the sense that there exists $\lambda_{\min}(P)\in\RPP$ such that the first inequality in \eqref{eq:lambd} below holds. For $P \in \mathcal{P}(\mathcal{H})$, we define the norm $\|x\|_P := \sqrt{\scal {x}{Px}}$. Note that $(\H,\scal{\cdot}{\cdot}_P)$ is a real Hilbert space and $\scal{x}{y} \leq \|x\|_{P^{-1}}\|y\|_{P}$ for every $(x,y) \in \H^2$. In addition, we denote $\lambda_{\min}(P)$ and $\lambda_{\max}(P)$ as the positive constants such that
	\begin{equation}\label{eq:lambd}
		(\forall x \in \H) \quad \lambda_{\min}(P)\|x\|^2 \le \|x\|_P^2 \le \lambda_{\max}(P)\|x\|^2. 
	\end{equation}
	Let $C \subset \mathcal{H}$ be a nonempty closed convex set and let $S \in \mathcal{P}(\mathcal{H})$. The distance from $x \in \mathcal{H}$ to $C$ with respect to $\|\cdot\|_S$ is $d_S(x,C) := \inf_{y \in C} \|x-y\|_S=\|x-P_C^S x\|_S$, where $P_C^S(x) := \arg\min_{y \in C} \|x-y\|_S$ is the metric projection of $x$ onto $C$.
	
	Let $A:\mathcal{H}\to 2^{\mathcal{H}}$ be a set-valued operator with graph $\gra A := \{(x,u) \in \mathcal{H}\times\mathcal{H} \mid u \in Ax\}$. The operator $A$ is said to be \emph{monotone} if $\scal {u-v} {x-y} \ge 0$ for every $((x,u),(y,v)) \in (\gra A)^2$, and it is maximally monotone if it is monotone and its graph is maximal in the sense of inclusions among the graphs of monotone operators. Given $P \in \mathcal{P}(\mathcal{H})$ and $\sigma>0$, $A$ is said to be \emph{$\sigma$-strongly monotone with respect to $\|\cdot\|_P$} if $\scal {u-v}{x-y }\ge \sigma \|x-y\|_P^2$ for every $((x,u),(y,v)) \in (\gra A)^2$. We say that $A$ is injective if, for every $(x,y) \in \H^2$, $Ax \cap Ay \neq\emptyset$ implies that $x = y$. The standard resolvent of $A$ is given by $J_A = (\id+A)^{-1}$. An extension of the classical resolvent is given by the {\it warped resolvent} \cite{BuiCombettesWarped2020}: given $D\subset \H$ nonempty, $U:D\to\H$, and $M:\H\to 2^\H$ such that $\ran U \subset \ran(U+M)$ and $U+M$ is injective, the \emph{warped resolvent} of $M$ with kernel $U$ is $J_M^U := (U+M)^{-1}\circ U$, that is, $p = J_M^U x$ if and only if $Ux - Up \in M p$. We refer the reader to \cite{BuiCombettesWarped2020,combettes2024geometry} for further properties and illustrative examples.
	
	Let $S \in \mathcal{P}(\mathcal{H})$,  $T:\mathcal{H}\to\mathcal{H}$, $\beta >0$, and $\zeta>0$. The operator $T$ is $\beta$-cocoercive with respect to $S$ if $\scal {Tx-Ty}{x-y}\ge\beta \|Tx-Ty\|_{S^{-1}}^2$ for every $(x,y) \in \mathcal{H}^2$, and it is $\zeta$-Lipschitz with respect to $S$ if $\|Tx-Ty\|_{S^{-1}} \le \zeta \|x-y\|_S$ for every $(x,y) \in \mathcal{H}^2$. In the case where $S= \id$, we simply say that the operator is $\beta$-cocoercive (or $\zeta$-Lipschitz, respectively).  
	
	Let $f:\H\to(-\infty,+\infty]$. The domain of $f$ is $\dom f=\{x\in\H \mid f(x)<+\infty\}$. Recall that $f$ is said to be proper if $\dom f\neq\emptyset$, and convex (respectively, lower semicontinuous (l.s.c.)) if $\epi f$ is a convex (respectively, closed) subset of $\H\times\R$. We denote by $\Gamma_0(\H)$ the class of all proper, convex, and lower semicontinuous functions on $\H$.  The subdifferential of $f$ is the set-valued operator $\partial f\colon x\mapsto\menge{v\in\H}{(\forall y\in\H)\ f(y)\geq f(x)+\scal {y-x}{v}}$. If $f \in \Gamma_0(\H)$, then $\partial f$ is a maximally monotone operator and $(\partial f)^{-1} = \partial f^*$, where $f^* \in \Gamma_0(\H)$ denotes the Fenchel conjugate of $f$. Furthermore, if $f\in\Gamma_0(\H)$ is differentiable and its gradient $\nabla f$ is $\beta$-Lipschitz continuous, it is $\frac{1}{\beta}$-cocoercive \cite[Corollary~18.17]{bauschkebook2017}.
	
	We conclude this section with a standard projection formula that will be used in the convergence analysis of the proposed algorithms. It follows from the usual projection formula onto a closed affine half-space applied to the Hilbert space
	$(\H,\scal{\cdot}{\cdot}_S)$; see for instance \cite[Example~29.20]{bauschkebook2017}.
	\begin{lem}[General Projection onto Affine half-space]
		\label{lem:general_projection}
		Let $\varphi: \mathcal{H} \to \mathbb{R}$ be a nonconstant affine function with $\nabla \varphi \neq 0$. 
		Let $H = \{ z \in \mathcal{H} : \varphi(z) \leq 0 \}$ be a half-space, and let $\|\cdot\|_S$ be a norm induced by some $S \in \mathcal{P}(\mathcal{H})$. 
		Then the projection of $x \in \mathcal{H}$ onto $H$ is given by:
		\[
		P^{S}_{H}(x) = x - \frac{\max\{0, \varphi(x)\}}{\|\nabla \varphi\|_{S^{-1}}^2} S^{-1} \nabla \varphi.
		\]
	\end{lem}
\section{Inexact Warped Resolvent iterations}\label{sec:main}
 %\subsection{The Inexact Nonlinear Forward--Backward Map}
 In the context of Problem~\ref{pro:main}, let $\bm{M} \colon \HH \to \HH$ and suppose that $(\bm{M}+\bm{C}+\bm{A})^{-1}$ is well-defined. Given $x \in \HH$, the warped resolvent step of the algorithm in \cite[Eq.~(4.34)]{combettes2024geometry} is given by
	\begin{equation}\label{eq:BFB}
		x_{+} = \bm{J}_{\bm A+ \bm{C}}^{\bm M}(x)=(\bm{M}+\bm{C}+\bm{A})^{-1}(\bm{M}x). 
	\end{equation}
	This step can be equivalently characterized by the inclusion-equation system \begin{align} \label{eq:appsol}
		\textnormal{find } \quad (v,x_+)\in \HH^2 \quad \textnormal{ such that } \quad \begin{cases}
			v \in \bm{A}x_+\\[1mm]
			v + \bm{M}x_+ - \bm{M}x + \bm{C}x_+ = 0.
		\end{cases}
	\end{align}
	Motivated by this formulation, and the relative-error framework introduced in \cite{Solodov1999SVA,Solodov2001}, we define the following notion of approximate solution to \eqref{eq:appsol}.
	\begin{defi}\label{def:ine-sol} Let $\bm{P} \in \mathcal{P}(\HH)$, a pair $(w, v) \in \HH \times \HH$ is called a \emph{$\sigma$-approximate solution} of the system \eqref{eq:appsol} if $\sigma \in [0,1[$ and
		\begin{align}\label{def:sigma-solution}
			\begin{cases}
				v \in \bm{A}w\\
				e=v + \bm{M}w - \bm{M}x + \bm{C}w,\\
				\|e\|_{\bm{P}^{-1}} \leq \sigma \|w - x\|_{\bm{P}}.
			\end{cases}
		\end{align}
	\end{defi}
	When $\sigma = 0$, condition \eqref{def:sigma-solution} reduces to the exact system \eqref{eq:appsol}, and hence to exact the warped resolvent step \eqref{eq:BFB}. In particular, if $\bm M=\Id$ and $\bm C=0$, given $x\in\HH$, $J_{\bm A}x$ is approximated by a pair $(w,v)$ such that $v\in\bm Aw$ and $e=v+w-x$ satisfies $\|e\|_{\bm P^{-1}}\le\sigma\|w-x\|_{\bm P}$, as in \cite{Solodov1999SVA,Solodov1999JCA}; if $e=0$, then $v=x-w\in\bm Aw$ and $w=J_{\bm A}x$. In this sense, the residual $e$ is controlled \emph{relative} to the displacement $w - x$, in the spirit of relative-error proximal methods. This notion of approximate solution naturally leads to the inexact warped resolvent algorithm. Before presenting the algorithm, we introduce the standing assumptions ensuring its well-definedness.
	
	\begin{asume}\label{asume:NFBE} In the context of Problem~\ref{pro:main}, consider the following assumptions.
		\begin{enumerate}
			\item\label{asume:NFBE1} Let $\bm{S} \in \mathcal{P}(\HH)$ and $\bm{P} \in \mathcal{P}(\HH)$.
			\item\label{asume:NFBE2} Suppose that $\bm{C}$ is $\beta$-cocoercive with respect to $\bm{P}$ for $\beta \in \RPP$.
			\item\label{asume:NFBE3} Let $(\lambda_n)_{n \in \N}$ be a sequence in $[\underline{\lambda},\overline{\lambda}]\subset~]0,2[$.
			\item \label{asume:NFBE4}For each $n \in \N$, let $\bm{M}_n\colon \HH \to \HH$ be an operator such that $\ran \bm{M}_n \subset \ran (\bm{M}_n+\bm{A}+\bm{C})$.% and $\bm{M}_n+\bm{A}+\bm{C}$ is injective. \textcolor{red}{ Hay que pensar un poco en esta condicion, quizas se puede relajar porque estamos considerando errores.}
		\end{enumerate}
	\end{asume}
	The next algorithm is the inexact warped resolvent version of the algorithm in \cite[Eq. (4.34)]{combettes2024geometry}.
	\begin{algo}\label{alg:NFBE} In the context of Problem~\ref{pro:main}  and Assumption~\ref{asume:NFBE}, let $x_0 \in \HH$, $\sigma \in [0,1[$, and consider the following recurrence.
		\begin{equation}\label{eq:algWR}
			(\forall n\in\N)\quad 
			\begin{array}{l}
				\left\lfloor
				\begin{array}{l}
					%w_n = (\bm{M}_n+\bm{A}+\bm{C})^{-1}(\bm{M}_nx_n) \Leftrightarrow w_n^* \in \bm{A} w_n \\
					\textnormal{find } (w_n,v_n) \in \gra \bm{A}  \textnormal{ such that }\\
					\left\lfloor
					\begin{array}{l} w_n^* = \bm{M}_nx_n-\bm{M}_nw_n-\bm{C}w_n\\
						e_n = v_n - w_n^*\\
						\|e_n\|_{\bm{P}^{-1}}\leq \sigma \|w_n-x_n\|_{\bm{P}}
					\end{array}
					\right.\\
					t_n^* = v_n+\bm{C}x_n\\
					\delta_n = \scal{x_n-w_n}{t_n^*}-\frac{1}{4\beta}\|w_n-x_n\|^2_{{\bm{P}}} \\
					d_n = \begin{cases}
						\dfrac{\delta_n}{\|t^*_n\|^2_{\bm{S}^{-1}}}  \bm{S}^{-1}t^*_n, &\textnormal{ if }
						\delta_n >0; \\
						0,  &\textnormal{ otherwise } 
					\end{cases}\\
					x_{n+1}=x_n-\lambda_n d_n.
				\end{array}
				\right.
			\end{array}
		\end{equation}
	\end{algo}
	
	\begin{rem}\label{rem:algNFBE}
		\begin{enumerate}
			\item\label{rem:algNFBE0}The condition $\ran \bm{M}_n \subset \ran (\bm{M}_n+\bm{A}+\bm{C})$ in Assumption~\ref{asume:NFBE}.\ref{asume:NFBE4} guarantees that the system in \eqref{def:sigma-solution} admits an exact solution. This ensures that the iterations of Algorithm~\ref{alg:NFBE} are well-defined. Unlike the frameworks in \cite{BuiCombettesWarped2020,combettes2024geometry}, we do not require the operator $\bm{M}_n+\bm{A}+\bm{C}$ to be injective.
			\item\label{rem:algNFBE1} In the case that $\bm{S} = \Id$ and $\sigma =0$, Algorithm~\ref{alg:NFBE} coincides with the algorithm in  \cite[Eq. (4.34)]{combettes2024geometry} with $q_n \equiv x_n$. Moreover, it reduces to the algorithm proposed in \cite{BuiCombettesWarped2020} if, additionally, we set $\bm{C}=0$.  On the other hand, if $\sigma =0$ and we set $\bm{M}_n = \widetilde{\bm M}_n-\bm{C}$, where $ \widetilde{\bm M}_n \colon \HH \to \HH$ is a strongly monotone operator, Algorithm~\ref{alg:NFBE} reduces to the nonlinear forward-backward (NFB) method with projection correction studied in \cite{Giselsson2021NFBS}. In particular, for $\gamma\in\,]0,2\beta[$ and $\bm M_n=\Id/\gamma-\bm C$, we have $w_n=(\Id/\gamma+\bm A)^{-1}(x_n/\gamma-\bm Cx_n)=J_{\gamma\bm A}(x_n-\gamma\bm Cx_n)$, which is the forward-backward (FB) step \cite{Chenhg1997,passty1979JMAA}. In this case, the explicit variant in Section~\ref{sec:explicit} (with $\bm S=\Id$ and $\gamma_n\equiv\gamma$) recovers FB when $\sigma=0$ and yields the inexact FB iteration $x_{n+1}=w_n-\gamma e_n$ when $\sigma>0$.
			\item\label{rem:algNFBE2} Algorithm~\ref{alg:NFBE} admits the following geometric interpretation as a projection method onto a sequence of separating half-spaces. For each $n\in\N$, define the half-space
			\begin{equation}\label{eq:H_n}
				\bm{H}_n=\menge{z \in \HH}{\scal{z-w_n}{t_n^*}\leq \frac{1}{4\beta}\|w_n-x_n\|^2_{\bm{P}}}.
			\end{equation}
			Then, the update $x_{n+1}$ coincides with a relaxed projection of $x_n$ onto $\bm{H}_n$ with respect to the metric induced by $\bm S$. Moreover, the construction of $(w_n,v_n)$ together with the relative-error condition ensures that the solution set $\mathcal S$ is contained in $\bm{H}_n$. Hence $\bm{H}_n$ separates the current iterate $x_n$ from $\mathcal S$ whenever $\delta_n>0$, as shown below. Consequently, Algorithm~\ref{alg:NFBE} can be viewed as a cutting-plane method that generates a sequence of outer approximations $(\bm{H}_n)_{n\in\N}$ of the solution set and performs (relaxed) projections onto these half-spaces. This perspective is closely related to the geometric framework of projection methods for monotone inclusions; see, e.g., \cite{combettes2001fejer,combettes2024geometry,Solodov1999JCA}.
		\end{enumerate}
	\end{rem}%
	\begin{teo}\label{teo:convergencia}
		In the context of Problem~\ref{pro:main}  and Assumption~\ref{asume:NFBE}, let $x_0 \in \HH$ and consider the sequence $(x_n)_{n \in \N}$ generated by Algorithm~\ref{alg:NFBE}. Then the following hold:
		\begin{enumerate}
			\item\label{h_k} $\zer(\bm A+\bm C) \subset \bm{H}_n$, where $\bm{H}_n$ is defined as in \eqref{eq:H_n}.
			%   \begin{equation}\label{eq:H_n}
				%   \bm{H}_n=\menge{z \in \HH}{\scal{z-w_n}{t_n^*}\leq \frac{1}{4\beta}\|w_n-x_n\|^2_{\bm{P}}}.
				%   \end{equation}
			
			\item\label{teo:convergencia0} For each $x\in \zer(\bm A+\bm C)$, the sequence $(\|x_n-x\|_{\bm S})_{n\in\N}$ converges.
			\item\label{teo:convergencia1} $\sum_{n \in \N}  \|d_n\|^2_{\bm{S}} < +\infty$ and hence $d_n \to 0$;
			\item\label{teo:convergencia2} Assume, in addition, that one of the following conditions is satisfied:
			\begin{enumerate}
				\item\label{teo:convergencia2a}  $w_n-x_n \to 0$ and $\bm{M}_nw_n-\bm{M}_nx_n \to 0$;
				\item\label{teo:convergencia2b}  $\bm{M}_n$ is $\alpha$-strongly monotone with respect to ${\bm{P}}$ for $\alpha \in ~]1/(4\beta)+\sigma,+ \infty[$  and $\zeta$-Lipschitz continuous for $\zeta \in \RPP$.
			\end{enumerate}
			Then, $(x_n)_{n \in \N}$ converges weakly to a point in $\mathcal{S}$.
		\end{enumerate}
	\end{teo}
	\begin{proof} Fix $n \in \N$. It follows from \eqref{eq:algWR} and Lemma~\ref{lem:general_projection} that 
		\begin{equation}\label{eq:proofteocon0}
			x_{n+1} = x_n + \lambda_n(P_{\bm{H}_n}^{\bm{S}}(x_n)-x_n).
		\end{equation}
		Moreover, for any $z \in \mathcal{S}$, we have $-\bm{C}z \in \bm{A}z$ and from \eqref{eq:algWR} that $v_n \in \bm{A} w_n$. Therefore, by the monotonicity of $\bm{A}$ we deduce that 
		\begin{equation}\label{eq:proofteocon1}
			\scal{v_n+\bm{C}z}{w_n-z}\geq 0.
		\end{equation}
		Furthermore, by \eqref{eq:proofteocon1}, the cocoercivity of $\bm{C}$ with respect to ${\bm{P}}$, and combining Cauchy--Schwarz and Young's inequalities, we deduce 
		\begin{align*}
			\scal{z-w_n}{t_n^*}  &= \scal{v_n+\bm{C}x_n}{z-w_n}\\
			& = \scal{v_n+\bm{C}x_n+\bm{C}z-\bm{C}z}{z-w_n}\\
			& = \scal{v_n+\bm{C}z}{z-w_n}+\scal{\bm{C}x_n-\bm{C}z}{z-w_n}\\
			& \leq \scal{\bm{C}x_n-\bm{C}z}{z-w_n}\\
			& = \scal{\bm{C}x_n-\bm{C}z}{z-x_n}+\scal{\bm{C}x_n-\bm{C}z}{x_n-w_n}\\
			& \leq -\beta \|\bm{C}x_n -\bm{C}z\|_{{\bm{P}}^{-1}}^2 + \frac{1}
			{4\beta}\|x_n-w_n\|^2_{\bm{P}}+\beta\|\bm{C}x_n-\bm{C}z\|^2_{{\bm{P}}^{-1}}\\
			& = \frac{1}{4\beta}\|x_n-w_n\|^2_{\bm{P}}.
		\end{align*}
		Then, $\mathcal{S} \subset \bm{H}_n$, proving \ref{h_k}. Set now $p_n=P_{\bm{H}_n}^{{{\bm{S}}}}(x_n)$. Then it follows from \eqref{eq:proofteocon0} and \cite[Theorem~3.16]{bauschkebook2017} that
		\begin{align}\label{eq:keyInq}
			\|x_{n+1}-z\|^2_{\bm{S}} &= \|x_n+\lambda_n(p_n -x_n)-z\|^2_{\bm{S}} \notag \\
			&= \|x_n- z\|^2_{\bm{S}}  + 2 \lambda_n\scal{x_n-z}{p_n -x_n}_{\bm{S}} +\lambda_n^2\| p_n -x_n\|^2_{\bm{S}}  \notag  \\
			&= \|x_n- z\|^2_{\bm{S}}  + 2 \lambda_n\scal{p_n-z}{p_n -x_n}_{\bm{S}} -\lambda_n(2-\lambda_n)\| p_n -x_n\|^2_{\bm{S}}  \notag \\
			&\leq \|x_n- z\|^2_{\bm{S}}  -\lambda_n(2-\lambda_n)\| p_n -x_n\|^2_{\bm{S}}  \notag \\
			&= \|x_n- z\|^2_{\bm{S}}  -\lambda_n(2-\lambda_n)\| d_n\|^2_{\bm{S}}.
		\end{align}
		Hence, \ref{teo:convergencia0} follows from \cite[Lemma~5.31]{bauschkebook2017} which in addition yields $\sum_{n \in \N} {\lambda}_n (2-{\lambda}_n)\|d_n\|^2_{\bm{S}} < +\infty$. Since $\inf {\lambda}_n >0$ and $\sup {\lambda}_n < 2$, we conclude \ref{teo:convergencia1}. 
		
		Suppose that \ref{teo:convergencia2a} holds. In particular, we have that $e_n \to 0$. Moreover, it follows from \eqref{eq:algWR} that
		\begin{equation}\label{eq:w_n}
			\bm{M}_nx_n-\bm{M}_nw_n-\bm{C}w_n+e_n = v_n \in \bm{A}w_n.
		\end{equation} 
		Hence, given a weak limit point $x^*$ of $(x_{n})_{n \in \N}$, say $x_{n_k} \weak x^*$, we have $w_{n_k}\weak x^*$ and $\bm{M}_{n_k}w_{n_k}-\bm{M}_{n_k}x_{n_k} +e_{n_k}\to 0$. Therefore, by the weak-strong closure of the maximally monotone operator $\bm{A}+\bm{C}$ \cite[Corollary~25.5 \& Proposition~20.38]{bauschkebook2017}, we conclude that  $0 \in (\bm{A}+\bm{C})x^*$ and the result follows from \cite[Lemma~2.47]{bauschkebook2017}. % 	{\color{blue}
		Assume that \ref{teo:convergencia2b} holds. We claim that
		\begin{equation}\label{eq:e_k}
			x_n-w_n \to 0
			\qquad\text{and}\qquad
			e_n\to 0.    
		\end{equation}
		First, from \ref{teo:convergencia1} we know that $d_n\to 0$. 
		Next, recall that
		\begin{equation*}
			t_n^*=v_n+\bm Cx_n
			=e_n+\bm M_nx_n-\bm M_nw_n-\bm Cw_n+\bm Cx_n.   
		\end{equation*}
		Hence, by the $\zeta$-Lipschitz continuity of $\bm M_n$, the
		$\frac1\beta$-Lipschitz continuity of $\bm C$ with respect to $\bm{P}$,  and the 
		%we obtain
		%\[
		%\|t_n^*\|
		%\le
		%\zeta\|x_n-w_n\|+\frac1\beta\|x_n-w_n\|+\|e_n\|.
		%\]
		%By 
		equivalence of norms \eqref{eq:lambd}, we have %there exists a constant $\lambda>0$ such that
		%\begin{equation}\label{eq:t_k}
		%\|t_n^*\|_{\bm S^{-1}}
		%\le
		%\kappa\Bigl(\zeta+\frac1\beta+\sigma\Bigr)\|x_n-w_n\|_{\bm P}.
		%\end{equation}
		%
		\begin{align}\label{eq:t_k}
			\|t_n^*\|_{\bm S^{-1}}
			&\le \sqrt{\lambda_{\max}(\bm S^{-1})}
			\left(
			\|e_n\|+\|\bm M_nx_n-\bm M_nw_n\|+\|\bm Cx_n-\bm Cw_n\|
			\right) \notag \\
			&\le
			\sqrt{\lambda_{\max}(\bm S^{-1})}
			\left(
			\frac{\beta\sigma+1}{\beta\sqrt{\lambda_{\min}(\bm P^{-1})}}
			+\frac{\zeta}{\sqrt{\lambda_{\min}(\bm P)}}%+\frac{1}{\beta\sqrt{\textcolor{blue}{\lambda_{\min}}(\bm P^{-1})}}
			\right)
			\|x_n-w_n\|_{\bm P}.
		\end{align}
		On the other hand, in view of \eqref{eq:algWR}, the $\beta$-cocoercivity of $\bm{C}$ with respect to ${\bm{P}}$, and the $\alpha$-strong monotonicity of $\bm{M}_n$, we deduce
		\begin{align}\label{eq:del_k}
			\delta_n &=\scal{x_n-w_n}{t_n^*}-\frac{1}{4\beta}\|x_n-w_n\|^2_{{\bm{P}}} \notag \\
			&=\scal{x_n-w_n}{e_n+\bm{M}_nx_n-\bm{M}_nw_n-\bm{C}w_n+\bm{C}x_n}-\frac{1}{4\beta}\|x_n-w_n\|^2_{{\bm{P}}}\notag \\
			& \geq\left(\alpha-\frac{1}{4\beta}\right)\|x_n-w_n\|^2_{\bm{P}} - \|x_n-w_n\|_{\bm{P}}\|e_n\|_{\bm{P}^{-1}}+\beta\|\bm{C}x_n-\bm{C}w_n\|^2_{{\bm{P}}^{-1}}\notag \\
			& \geq \left(\alpha-\frac{1}{4\beta}\right)\|x_n-w_n\|^2_{\bm{P}}-\sigma\|x_n-w_n\|^2_{\bm{P}}\notag \\
			& = \left(\alpha-\frac{1}{4\beta}-\sigma\right)\|x_n-w_n\|^2_{\bm{P}}.
		\end{align} 
		%Since $\alpha>\sigma+\frac1{4\beta}$, it follows that there exists $\mu>0$ such that
		%\begin{equation}\label{eq:del_k}
		% \delta_n\ge \mu \|x_n-w_n\|_{\bm P}^2.   
		%\end{equation}
		%
		In particular, if $\delta_n=0$ we have $x_n-w_n=0$. Otherwise, if $\delta_n>0$, from \eqref{eq:algWR}
		%$$
		%d_n=\frac{\delta_n}{\|t_n^*\|_{\bm S^{-1}}^2}\,\bm S^{-1}t_n^*,
		%$$
		%and therefore
		\begin{equation}\label{eq:d_k}
			\|d_n\|_{\bm S} = \frac{\delta_n}{\|t_n^*\|_{\bm S^{-1}}}.   
		\end{equation}
		%{\color{blue}(If $\delta_n\le 0$, then necessarily $d_n=0$, and the estimate below is trivial)}.
		Combining \eqref{eq:t_k}, \eqref{eq:del_k}, and \eqref{eq:d_k}, we obtain (the case $\delta_n=0$ is direct, since then $d_n=0$ and $x_n=w_n$)
		\begin{equation}\label{eq:d_S}
			\|d_n\|_{\bm S} = \frac{\delta_n}{\|t_n^*\|_{\bm S^{-1}}}
			\ge \frac{\alpha-\frac{1}{4\beta}-\sigma}{\sqrt{\lambda_{\max}(\bm S^{-1})}\left(
				{\frac{(\sigma+1/\beta)}{\sqrt{\lambda_{\min}(\bm P^{-1})}}
					+\frac{\zeta}{\sqrt{\lambda_{\min}(\bm P)}}}
				\right)}\|x_n-w_n\|_{\bm P}.   
		\end{equation}
		Since $d_n\to 0$ by \ref{teo:convergencia1}, and $\alpha>\sigma+\frac1{4\beta}$, it follows that $\|x_n-w_n\|_{\bm P}\to 0$, and hence $x_n-w_n\to 0$. Moreover, from the relative-error condition in Algorithm~\ref{alg:NFBE}, we also conclude that $e_n\to 0$. Finally, by the $\zeta$-Lipschitz continuity of $\bm M_n$,
		$$
		\|\bm M_nw_n-\bm M_nx_n\| \le \zeta\|w_n-x_n\|\to 0.
		$$
		Thus condition \ref{teo:convergencia2a} holds, and the conclusion follows from the previous step.
	\end{proof}
	%Theorem~\ref{teo:convergencia} can be viewed as an inexact, metric-induced version of \cite[Theorem~4.8]{combettes2024geometry}.
	\begin{rem}\label{rem:teomain} \begin{enumerate}
			\item\label{rem:teomain1} In view of Remark~\ref{rem:algNFBE}.\ref{rem:algNFBE1},
			Theorem~\ref{teo:convergencia} is closely related to \cite[Theorem~4.8]{combettes2024geometry} and \cite[Theorem~5.1]{Giselsson2021NFBS}, and can be interpreted as their inexact, relative-error counterpart. Moreover, if $\bm{C}=0$, Algorithm~\ref{alg:NFBE} reduces to a variable-metric version of \cite[Algorithm~1.1]{Solodov1999JCA}. Consequently, Theorem~\ref{teo:convergencia} can be viewed as a variable-metric extension of the corresponding convergence results in \cite{Solodov1999JCA}.
			\item\label{rem:teomain2} Note that under assumption \ref{teo:convergencia2b} of Theorem~\ref{teo:convergencia}, $J_{\bm{A}+\bm{C}}^{\bm M_n}$ has full domain and it is single-valued for each $n \in \N$ \cite[Proposition~4.1]{Giselsson2021NFBS}. Therefore, the system in \eqref{def:sigma-solution} admits an exact solution and Algorithm~\ref{alg:NFBE} is well-defined. In particular, Assumption~\ref{asume:NFBE}.\ref{asume:NFBE4} holds directly.
			\item Suppose that condition \ref{teo:convergencia2b} of Theorem~\ref{teo:convergencia} holds. If there exists $m \in \N$ such that $\delta_m \leq 0$, it follows from \eqref{eq:del_k} that $x_m = w_m$. Furthermore, from \eqref{eq:algWR}, it follows that $e_m = 0$, $w_m^* = v_m = -\bm Cx_m$, and $-\bm{C}x_m \in \bm{A}x_m$. Therefore, $x_m \in \mathcal{S}$ and $x_n = x_m$ for every $n \geq m$. Hence, the condition $\delta_m \leq 0$ can be used as a stopping criterion for Algorithm~\ref{alg:NFBE}. %\textcolor{blue}{Note that, without condition \ref{teo:convergencia2b} (e.g., if $\bm M_n=0$), $\delta_m\le 0$ does not imply $x_m\in\mathcal S$ in general; in that case, $d_m=0$ and $x_{m+1}=x_m$, and the analysis of Theorem~\ref{teo:convergencia} still applies. The analogous property for the explicit variant in Section~\ref{sec:explicit} is established in \eqref{eq:deltapos} below.}
		\end{enumerate} 
	\end{rem}
	\subsection{Inexact warped resolvent with explicit steps}\label{sec:explicit} In this section, we derive an explicit version of Algorithm~\ref{alg:NFBE} by imposing the following additional assumptions on the operators $\bm{C}$ and $\bm{M}_n$.
	\begin{asume}\label{asume:NFBEE} In the context of Problem~\ref{pro:main}, let $\bm{S}\in \mathcal{P}(\HH)$, let $\gamma \in \RPP$, suppose that there exists $(\gamma_n)_{n \in \N} \subset [\gamma,+\infty[$ such that $\gamma_n (\bm{M}_n+\bm{C})-{\bm{S}}$ is $\zeta_n$-Lipschitz with respect to $\bm{S}$ for $\zeta_n \in [0,\zeta]$ and $\zeta \in~]0,1[$. Moreover,  suppose that $\bm{C}$ is $\beta$-cocoercive with respect to $\bm{S}$ for $\beta \in \RPP$.
	\end{asume}
	In order to reformulate Algorithm~\ref{alg:NFBE} with explicit steps, we set $\bm P=\bm S$ and, whenever $\delta_n>0$, we choose $\lambda_n=\gamma_n\|t_n^*\|^2_{\bm S^{-1}}/\delta_n$, where $t_n^*$ and $\delta_n$ are defined as in \eqref{eq:algWR}. In that case, $\lambda_nd_n=\gamma_n\bm S^{-1}t_n^*$. Moreover, as shown in the proof of Theorem~\ref{teo:NFBEE}, under the assumptions of that theorem the case $\delta_n\le 0$ only occurs when $x_n=w_n\in\mathcal S$ and $t_n^*=0$. With this choice, Algorithm~\ref{alg:NFBE} reduces to the following explicit form.
	\begin{algo}\label{alg:NFBEE} In the context of Problem~\ref{pro:main}  and Assumption~\ref{asume:NFBEE}, let $x_0 \in \HH$, $\sigma \in [0,1[$, and consider the following recurrence.
		\begin{equation}\label{eq:algWRE}
			(\forall n\in\N)\quad 
			\begin{array}{l}
				\left\lfloor
				\begin{array}{l}
					%w_n = (\bm{M}_n+\bm{A}+\bm{C})^{-1}(\bm{M}_nx_n) \Leftrightarrow w_n^* \in \bm{A} w_n \\
					\textnormal{ find } (w_n,v_n) \in \gra\bm{A}
					\textnormal{ such that }\\
					\left\lfloor
					\begin{array}{l} w_n^* = \bm{M}_nx_n-\bm{M}_nw_n-\bm{C}w_n,\\
						e_n = v_n - w_n^*,\\        
						\|e_n\|_{\bm{S}^{-1}}\leq \sigma \|w_n-x_n\|_{\bm{S}},	
					\end{array}
					\right.\\
					%		e_n = v_n - w_n^*\\
					%	\|e_n\|_{\bm{S}^{-1}}\leq \sigma \|w_n-x_n\|_{\bm{S}}\\
					%t_n^* = v_n+\bm{C}x_n\\
					%\delta_n = \scal{x_n-w_n}{t_n^*}-\frac{1}{4\beta}\|w_n-x_n\|^2_{{\bm{P}}} \\
					%d_n = \begin{cases}
						%	\dfrac{\delta_n}{\|t^*_n\|^2_{\bm{S}^{-1}}}  \bm{S}^{-1}t^*_n, &\textnormal{ if }
						%	\delta_n >0; \\
						%		0,  &\textnormal{ otherwise } 
						%	\end{cases}\\
					x_{n+1}=x_n-\gamma_n \bm{S}^{-1}(v_n+\bm{C}x_n).
				\end{array}
				\right.
			\end{array}
		\end{equation}
	\end{algo}
	\begin{rem}\label{rem:algexplicit}
		\begin{enumerate}
			\item\label{rem:algexplicit1} In view of \cite[Proposition~2.1]{MorinBanertGiselsson2022}, Assumption~\ref{asume:NFBEE} guarantees that $\bm M_n+\bm C$ is $(1-\zeta_n)/\gamma_n$-strongly monotone with respect to $\bm{S}$, for every $n \in \N$. Then,  $J_{\bm{A}+\bm{C}}^{\bm M_n}$ has full domain and it is single-valued \cite[Proposition~4.1]{Giselsson2021NFBS}. Therefore, the inclusion in \eqref{eq:algWRE}  admits a solution for $e_n = 0$ and Algorithm~\ref{alg:NFBEE} is well-defined. 
			\item\label{rem:algexplicit2} Consider Algorithm~\ref{alg:NFBEE} when $\bm{C} = 0$, $\bm{M}_n=\bm{S}/\gamma_n$. Define $v_n = \bm{S}\widetilde{v}_n$. Hence, \eqref{eq:algWRE} reduces to
			\begin{equation}
				(\forall n\in\N)\quad 
				\begin{array}{l}
					\left\lfloor
					\begin{array}{l}
						%w_n = (\bm{M}_n+\bm{A}+\bm{C})^{-1}(\bm{M}_nx_n) \Leftrightarrow w_n^* \in \bm{A} w_n \\
						\textnormal{ find } (w_n,\widetilde{v}_n) \in \HH \times \HH \textnormal{ such that }\\
						\left\lfloor
						\begin{array}{l} 
							\bm{S}\widetilde{v}_n \in \bm{A} w_n\\
							\| \gamma_n \widetilde{v}_n -x_n+w_n\|_{\bm{S}}\leq \gamma_n \sigma \|w_n-x_n\|_{\bm{S}},	
						\end{array}
						\right.\\
						x_{n+1}=x_n-\gamma_n \widetilde{v}_n.
					\end{array}
					\right.
				\end{array}
			\end{equation}
			This scheme was studied in \cite{AlvesLorenzNaldi2026} when the operator $\bm{S}$ is self-adjoint and positive semidefinite. In that context, the authors derive inexact versions of Chambolle--Pock \cite{chambolle2016AN} and Davis--Yin \cite{DavisYin2017} algorithms. Although this structure arises as an instance of Algorithm~\ref{alg:NFBEE}, we restrict our analysis to the case where $\bm{S}$ is positive definite.
		\end{enumerate}
	\end{rem}
	\begin{teo}\label{teo:NFBEE}
		In the context of Problem~\ref{pro:main}  and Assumption~\ref{asume:NFBEE}, let $x_0 \in \HH$ and consider the sequence $(x_n)_{n \in \N}$ generated by Algorithm~\ref{alg:NFBEE}.
		Moreover, suppose that there exist $\varepsilon \in \RPP$ and $\delta \in ]0,1[$ such that \begin{equation}\label{eq:stepsizes}
			(\forall n \in \N) \quad  1-\varepsilon-(\zeta_n+\gamma_n\sigma)^2 \geq \delta \quad \textnormal{ and } \quad 2\beta \varepsilon \geq \gamma_n.
		\end{equation}
		Then, $(x_n)_{n \in \N}$ converges weakly to a point in $\mathcal{S}$.
	\end{teo}
	\begin{proof}
		{For each $n \in \N$, define $t_n^*$, $\delta_n$, and $d_n$ as in \eqref{eq:algWR} with $\bm{P}=\bm{S}$. We first show that
			\begin{equation}\label{eq:deltapos}
				(\forall n\in\N)\quad \delta_n\ge \frac{\delta}{2\gamma_n}\|x_n-w_n\|^2_{\bm S}.
			\end{equation}
			Indeed, fix $n\in\N$. By the $(1-\zeta_n)/\gamma_n$-strong monotonicity of $\bm M_n+\bm C$ with respect to $\bm S$ (Remark~\ref{rem:algexplicit}.\ref{rem:algexplicit1}), the Cauchy--Schwarz inequality, and the relative-error condition in \eqref{eq:algWRE}, we have
			\begin{align*}
				\delta_n &= \scal{x_n-w_n}{(\bm M_n+\bm C)x_n-(\bm M_n+\bm C)w_n+e_n}-\frac{1}{4\beta}\|x_n-w_n\|^2_{\bm S}\\
				&\ge \frac{1}{\gamma_n}\left(1-\zeta_n-\gamma_n\sigma-\frac{\gamma_n}{4\beta}\right)\|x_n-w_n\|^2_{\bm S}.
			\end{align*}
			Moreover, \eqref{eq:stepsizes} yields $\zeta_n+\gamma_n\sigma\le\sqrt{1-\varepsilon-\delta}<1$ and $\gamma_n/(4\beta)\le\varepsilon/2$. Hence,
			\begin{equation*}
				1-\zeta_n-\gamma_n\sigma-\frac{\gamma_n}{4\beta}\ge \frac{1-(\zeta_n+\gamma_n\sigma)^2}{1+\zeta_n+\gamma_n\sigma}-\frac{\varepsilon}{2}\ge\frac{\varepsilon+\delta}{2}-\frac{\varepsilon}{2}=\frac{\delta}{2},
			\end{equation*}
			and \eqref{eq:deltapos} follows. Consequently, if $\delta_n\le 0$, then $x_n=w_n$, $e_n=0$, and $t_n^*=v_n+\bm Cx_n=w_n^*+\bm Cx_n=0$, which implies that $x_n\in\mathcal S$ (see Remark~\ref{rem:teomain}) and that $x_{n+1}=x_n$ in \eqref{eq:algWRE}. Now, for every $n\in\N$, set
			\begin{equation} \label{eq:deflam}
				\lambda_n = \begin{cases}
					\gamma_n \dfrac{\|t_n^*\|^2_{{\bm{S}}^{-1}}}{\delta_n}, &\text{if } \delta_n > 0;\\
					1, &\text{otherwise}.
				\end{cases} 
			\end{equation} 
			Note that $t_n^*\neq 0$ whenever $\delta_n>0$. Hence, in both cases of \eqref{eq:deflam}, we have $\lambda_nd_n=\gamma_n\bm S^{-1}t_n^*$, and Algorithm~\ref{alg:NFBEE} is a particular instance of Algorithm~\ref{alg:NFBE} with $\bm P=\bm S$.} Let us now verify the conditions of Theorem~\ref{teo:convergencia}.\ref{teo:convergencia2a} to conclude the result.  First, we prove that $\inf \lambda_n > 0$ and $\sup \lambda_n <2$. Fix $n \in \N$ such that $\delta_n>0$. By the strong monotonicity of $\bm{M}_n+\bm{C}$ (Remark~\ref{rem:algexplicit}.\ref{rem:algexplicit1}), we have
		\begin{align*}
			\lambda_n= \gamma_n \dfrac{\|t_n^*\|^2_{{\bm{S}}^{-1}}}{\delta_n}&\geq \gamma_n  \dfrac{\|t_n^*\|^2_{{\bm{S}}^{-1}}}{\scal{x_n-w_n}{t_n^*}}\\
			&\geq \gamma_n  \dfrac{\|t_n^*\|_{{\bm{S}}^{-1}}}{\|x_n-w_n\|_{{\bm{S}}}}\\
			&=\gamma_n  \dfrac{\|e_n+(\bm{M}_n+\bm{C})x_n-(\bm{M}_n+\bm{C})w_n\|_{{\bm{S}}^{-1}}}{\|x_n-w_n\|_{{\bm{S}}}}\\
			&\geq\gamma_n  \dfrac{\|(\bm{M}_n+\bm{C})x_n-(\bm{M}_n+\bm{C})w_n\|_{{\bm{S}}^{-1}}-\|e_n\|_{\bm{S}^{-1}}}{\|x_n-w_n\|_{{\bm{S}}}}\\
			&\geq 1-\zeta_n-\gamma_n\sigma\\
			&\geq 1-\sqrt{1-\delta}.
		\end{align*}
		Therefore, $\inf \lambda_n %\geq \min\{1, 1 -\sqrt{1-\delta}\} = 
		1 -\sqrt{1-\delta}> 0$. On the other hand, the $(2/\gamma_n)$-Lipschitz continuity of $\bm{M}_n+\bm{C}$ \cite[Proposition~2.1]{MorinBanertGiselsson2022} and \eqref{eq:algWRE} yield
		%\cite[Proposition~2.1]{MorinBanertGiselsson2022} that $\bm{M}_n+\bm{C}$ is $(2/\gamma_n)$-Lipschitz with respect to $\bm{S}$. Hence,
		\begin{equation}\label{eq:proofcon2d1}
			\scal{x_n-w_n}{t_n^*} %= \scal{x_n-w_n}{(\bm{M}_n+\bm{C})x_n-(\bm{M}_n+\bm{C})w_n+e_n} 
			\leq \frac{2+\gamma_n\sigma}{\gamma_n}\|x_n-w_n\|^2_{{\bm{S}}}.
		\end{equation}
		Then, it follows from \eqref{eq:proofcon2d1}, \eqref{eq:stepsizes}, the Lipschitzian property of $(\gamma_n\bm{M}_n+\gamma_n\bm{C}-{\bm{S}})$, and \eqref{eq:algWRE} that
		
		\begin{align*}
			&(1-\varepsilon)\|x_n-w_n\|_{\bm S}^2
			-\frac{\gamma_n\delta}{2+\gamma_n\sigma}
			\scal{x_n-w_n}{t_n^*}  \\
			&\geq (1-\varepsilon-\delta)\|x_n-w_n\|_{\bm S}^2\\
			&\geq (\zeta_n+\gamma_n\sigma)^2\|x_n-w_n\|_{\bm S}^2\\
			&\geq
			\Big(
			\|(\gamma_n\bm M_n+\gamma_n\bm C-\bm S)x_n
			-(\gamma_n\bm M_n+\gamma_n\bm C-\bm S)w_n\|_{\bm S^{-1}}
			+\gamma_n\|e_n\|_{\bm S^{-1}}
			\Big)^2\\
			&\geq
			\|\gamma_ne_n
			+(\gamma_n\bm M_n+\gamma_n\bm C-\bm S)x_n
			-(\gamma_n\bm M_n+\gamma_n\bm C-\bm S)w_n
			\|_{\bm S^{-1}}^2\\
			&=
			\|\gamma_n t_n^*-\bm S(x_n-w_n)\|_{\bm S^{-1}}^2\\
			&=
			\|\gamma_n t_n^*\|_{\bm S^{-1}}^2
			-2\gamma_n\langle x_n-w_n,t_n^*\rangle
			+\|x_n-w_n\|_{\bm S}^2,
		\end{align*}
		which implies
		\begin{equation}\label{eq:proofcon2d2}
			\left(2-\frac{\delta}{2+\gamma_n\sigma}\right)\gamma_n\scal{x_n-w_n}{t_n^*}-\varepsilon\|x_n-w_n\|^2_{{\bm{S}}}\geq\|\gamma_n t_n^*\|^2_{{\bm{S}}^{-1}}.
		\end{equation}
		Moreover, since ${\varepsilon}/{2}\geq {\gamma_n}/{(4\beta)}$  (see \eqref{eq:stepsizes}) and $ -\left(2-{\delta}/{(2+\gamma_n\sigma)}\right)  > - 2 $, we have 
		\begin{align}\label{eq:proofcon2d3}
			\left(2-\frac{\delta}{2+\gamma_n\sigma}\right)\gamma_n\delta_n&=
			\left(2-\frac{\delta}{2+\gamma_n\sigma}\right)\left(\gamma_n\scal{x_n-w_n}{t_n^*}-\frac{\gamma_n}{4\beta}\|x_n-w_n\|^2_{{\bm{S}}}\right)\nonumber\\
			& \geq   
			\left(2-\frac{\delta}{2+\gamma_n\sigma}\right)\left(\gamma_n\scal{x_n-w_n}{t_n^*}- \frac{\varepsilon}{2}\|x_n-w_n\|^2_{{\bm{S}}}\right)\nonumber\\
			&>    
			\left(2-\frac{\delta}{2+\gamma_n\sigma}\right)\gamma_n\scal{x_n-w_n}{t_n^*}- \varepsilon\|x_n-w_n\|^2_{{\bm{S}}}.
		\end{align}
		Combining \eqref{eq:proofcon2d2} and \eqref{eq:proofcon2d3}, we obtain
		\begin{equation}\label{eq:lambdam21}
			\left(2-\frac{\delta}{2+\gamma_n\sigma}\right)\delta_n >\gamma_n \| t_n^*\|_{\bm S^{-1}}^2.
		\end{equation}
		Moreover, since $0\le\gamma_n\sigma\le\zeta_n+\gamma_n\sigma\le\sqrt{1-\varepsilon-\delta}<1$ by \eqref{eq:stepsizes}, we deduce that
		\begin{equation}\label{eq:lambdam22}
			-\frac{\delta}{2+\gamma_n\sigma} \leq -\frac{\delta}{3}.
		\end{equation}
		Hence, \eqref{eq:lambdam21} and \eqref{eq:lambdam22} yield
		\begin{equation*} \lambda_n = \frac{\gamma_n \|t_n^*\|_{\bm S^{-1}}^2}{\delta_n}< 2-\frac{\delta}{2+\gamma_n\sigma}
			\leq 2-\frac{\delta}{3}.
		\end{equation*}
		Therefore,
		\begin{equation*}
			\sup\lambda_n\leq \max\{1,2-\delta/3\}<2.
		\end{equation*}
		Now, by the $(1-\zeta_n)/\gamma_n$-strong monotonicity of $(\bm{M}_n+\bm{C})$ with respect to $\bm{S}$, from \eqref{eq:algWR} and \eqref{eq:deflam} we have, for every $n\in\N$, %$\lambda_n d_n = \gamma_n {\bm{S}}^{-1}t_n^*$ and
		\begin{align*}
			\lambda_n\|d_n\|_{{\bm{S}}} &= \gamma_n\|{\bm{S}}^{-1}t_n^*\|_{{\bm{S}}} \\
			&= \gamma_n\|{\bm{S}}^{-1}(e_n+(\bm{M}_n+\bm{C})x_n-(\bm{M}_n+\bm{C})w_n)\|_{{\bm{S}}} \\
			&= \gamma_n\|e_n+(\bm{M}_n+\bm{C})x_n-(\bm{M}_n+\bm{C})w_n\|_{{\bm{S}}^{-1}} \\
			&\geq \gamma_n\|(\bm{M}_n+\bm{C})x_n-(\bm{M}_n+\bm{C})w_n\|_{{\bm{S}}^{-1}}-\gamma_n\|e_n\|_{{\bm{S}}^{-1}} \\
			&\geq (1-\zeta_n-\gamma_n\sigma)\|x_n-w_n\|_{\bm{{\bm{S}}}}\\ 
			&\geq \left(1-\sqrt{1-\delta}\right)\|x_n-w_n\|_{\bm{{\bm{S}}}}. 
		\end{align*}
		Therefore, since $\sup \lambda_n <2$, it follows from Theorem~\ref{teo:convergencia}.\ref{teo:convergencia1} that $x_n-w_n\to 0$. Moreover, by the $(2/\gamma_n)$-Lipschitz continuity of $(\bm{M}_n+\bm{C})$ with respect to ${\bm{S}}$ and the $\beta$-cocoercivity of $\bm{C}$ with respect to ${\bm{S}}$, we conclude that $\bm{M}_n$ is $(2/\gamma_n+1/\beta)$-Lipschitz with respect to ${\bm{S}}$. Then, $\bm{M}_nx_n-\bm{M}_nw_n \to 0$. The result follows from Theorem~\ref{teo:convergencia}.\ref{teo:convergencia2a}.
	\end{proof}
\subsection{Strong and linear convergence}
In this section, we study the strong convergence of Algorithm~\ref{alg:NFBE}. In particular, under a metric subregularity assumption on the operator $\bm A + \bm C$, we establish local linear convergence of the generated sequence. In addition, we show how strong convergence can be enforced by incorporating projection steps onto intersections of half-spaces via a Haugazeau-type scheme \cite{haugazeau1968inequations}.
\subsubsection{Strong convergence under metric subregularity}
The following definition introduces the notion of {\it metric subregularity}, which will be instrumental in establishing linear convergence results.
\begin{defi}\label{def:msr} \emph{(}\cite{wang2023convergence}\emph{)}
Let $\bm{F}:\HH \to 2^\HH$ be a set-valued operator and let
$\bm{S}\in\mathcal P(\HH)$. We say that $\bm{F}$ is \emph{metrically subregular at} 
$\bar x\in\zer \bm{F}$ %\emph{for} $\bar u\in \bm{F}\bar x$, % with respect to the $\bm{S}$,
 if there exist $\kappa>0$ and $\rho>0$ such that
\begin{equation}\label{eq:msr}
(\forall x\in B_{\bm{S}}(\bar x,\rho)) \quad d_{\bm{S}}(x,\zer \bm{F})
\;\le\;
\kappa\,d_{\bm{S}^{-1}}(0,\bm{F}x).
\end{equation}
We say that $\bm{F}$ is \emph{metrically subregular} on $\zer \bm{F}$  if there exist $\kappa>0$ and $\rho>0$ such that
\begin{equation}\label{eq:msr1}
(\forall \overline{x} \in \zer \bm{F})(\forall x\in B_{\bm{S}}(\bar x,\rho)) \quad d_{\bm S}(x,\zer \bm{F})
\;\le\;
\kappa\,d_{\bm S^{-1}}(0,\bm Fx).
\end{equation}
%where
%\begin{equation*}
%B_S(\mathcal S,\rho)
%:=\{x\in\HH:d_S(x,\mathcal S)\le\rho\}.
%\end{equation*}
\end{defi}
%\cite{wang2023convergence,dontchev2009implicit}
\begin{rem}\label{rem:stronglymonotone}
If $\bm{F}\colon \HH \to 2^{\HH}$ is $\eta$-strongly monotone with respect to $\bm{S} \in \mathcal{P}(\HH)$,
%, i.e.,
%\begin{equation*}
%\langle u-v,\ x-y\rangle \;\ge\; \eta\,\|x-y\|_S^2,
%\qquad\forall u\in F(x),\ v\in F(y),
%\end{equation*}
then $\bm{F}$ is metrically subregular on $\zer \bm{F}$ for $\kappa = 1/\eta$ and every $\rho \in \RPP$.
%, and \eqref{eq:msr} holds globally, i.e., with $\rho=+\infty$.
\end{rem}
\begin{teo}[Local linear convergence of Algorithm~\ref{alg:NFBE}]
\label{thm:linear_NFBE} In the context of Problem~\ref{pro:main}, let $
\bm{F}:=\bm A+\bm C$ assume that $\bm{F}$ is metrically subregular on $\zer \bm F=\mathcal{S}$ with parameters $\kappa>0$ and $\rho>0$. Let $x_0 \in \HH$ and let $(x_n)_{n \in \N}$ be a sequence generated in one of the following scenarios:
\begin{enumerate}
    %\item\label{thm:linear_NFBE1} Under Assumption~\ref{asume:NFBE}, condition  \ref{teo:convergencia2b} in Theorem~\ref{teo:convergencia}, and $(x_n)_{n\in \N}$ generated by Algorithm~\ref{alg:NFBE}.
        \item\label{thm:linear_NFBE1} $(x_n)_{n\in \N}$ is generated by Algorithm~\ref{alg:NFBE} under Assumption~\ref{asume:NFBE} and  condition  \ref{teo:convergencia2b} in Theorem~\ref{teo:convergencia}.
    \item\label{thm:linear_NFBE2}  $(x_n)_{n\in \N}$ is generated by Algorithm~\ref{alg:NFBEE} under Assumption~\ref{asume:NFBEE} and the setting of Theorem~\ref{teo:NFBEE}. 
\end{enumerate}
Moreover, suppose that there exists $n_0 \in \N$ and $r \in ~]0,\rho[$ such that
\begin{equation}\label{eq:lin_rate_NFBE_}
%\|x_{n_0}-\bar x \|\leq r.\\
 d_{\bm S}(x_{n_0},\mathcal S)\le r.
\end{equation}
Then, there exists $\tau\in~]0,1[$ satisfying
\begin{equation}\label{eq:lin_rate_NFBE}
(\forall n\ge n_0) \qquad d_{\bm S}(x_{n+1},\mathcal S)^2
\;\le\;
(1-\tau)\,d_{\bm S}(x_n,\mathcal S)^2,
\end{equation}
Consequently, \((x_n)_{n\in\N}\) converges strongly, and locally at linear rate, to some \(\bar x\in\mathcal S\).
\end{teo}

\begin{proof} \ref{thm:linear_NFBE1}
By Theorem~\ref{teo:convergencia}, for every $x\in\mathcal S$ the sequence
$(\|x_n-x\|_{\bm S})_{n\in\N}$ is convergent and
$$
\sum_{n\in\N}\|d_n\|_{\bm S}^2<+\infty.
$$
Moreover, under condition~\ref{teo:convergencia2b}, by \eqref{eq:d_S} (see Theorem~\ref{teo:convergencia}), we have
\begin{equation}\label{eq:key_dn}
(\forall n\in\N) \quad \|d_n\|_{\bm S}
\;\ge\;\mu_\sigma\|x_n-w_n\|_{\bm P}
\end{equation}
with 
\begin{equation}\label{eq:defmusigma}
\mu_\sigma:= \frac{\alpha-\frac{1}{4\beta}-\sigma}{\sqrt{\lambda_{\max}(\bm S^{-1})}\left(
{\frac{(\sigma+1/\beta)}{\sqrt{\lambda_{\min}(\bm P^{-1})}}
+\frac{\zeta}{\sqrt{\lambda_{\min}(\bm P)}}}
\right)}
\end{equation}
In particular,
\begin{equation}\label{eq:xw_to_0}
x_n-w_n\to 0
~~\text{and}~~
e_n\to 0~~\text{as}~ ~n\to + \infty.
\end{equation}
Fix $r\in]0,\rho[$ and assume that $ d_{\bm S}(x_{n_0},\mathcal S)\le r$. Since $(x_n)$ is Fej\'er monotone with respect to $\mathcal S$, %in the $\bm S$-metric, 
we have
\begin{equation*}
    (\forall n\geq n_0) \qquad d_{\bm S}(x_n,\mathcal S)\le d_{\bm S}(x_{n_0},\mathcal S)\le r.
\end{equation*}
Furthermore, by \eqref{eq:xw_to_0}, we can assume (possibly increasing $n_0$) that
\begin{equation*}
(\forall n\ge n_0) \qquad \|x_n-w_n\|_{\bm S}< \rho-r.
\end{equation*}
Therefore, 
\begin{equation*}
(\forall n\ge n_0)  \qquad d_{\bm S}(w_n,\mathcal S)
\le
\|w_n-x_n\|_{\bm S}+d_{\bm S}(x_n,\mathcal S)
<
\rho-r+r
=
\rho.
\end{equation*}
Hence $
w_n\in B_{\bm S}(\mathcal S,\rho)$ for all $n\ge n_0$,  so that the metric subregularity estimate \eqref{eq:msr} applies at $w_n$. On the other hand, since $v_n\in \bm A w_n$, we have
\begin{equation}\label{eq:unestrella}
u_n^*:=v_n+\bm Cw_n \in \bm{A}w_n + \bm{C}w_n=\bm{F}w_n.
\end{equation}
Moreover, from \eqref{eq:algWR}, we have
\begin{equation*}
u_n^*
=
v_n+\bm Cw_n
=
w_n^*+e_n+\bm Cw_n
=
\bm M_nx_n-\bm M_nw_n+e_n.
\end{equation*}
Since $\bm M_n$ is $\zeta$-Lipschitz continuous and
$\|e_n\|_{\bm P^{-1}}\le \sigma \|x_n-w_n\|_{\bm P}$, 
the equivalence of norms \eqref{eq:lambd} yields that
\begin{equation}\label{eq:un_bound}
(\forall n\ge n_0) \qquad \|u_n^*\|_{\bm S^{-1}}
\le
c_0\,\|x_n-w_n\|_{\bm P},
\end{equation}
where
\begin{equation*}
c_0=\sqrt{\lambda_{\max}(\bm{S}^{-1})}\left(\frac{\zeta}{\sqrt{\lambda_{\min}(\bm P)}}+ \frac{\sigma}{\sqrt{\lambda_{\min}(\bm P^{-1})}}\right) . 
\end{equation*}
Applying \eqref{eq:msr} at $w_n$ and by \eqref{eq:unestrella} and \eqref{eq:un_bound}, we obtain
\begin{equation*}
d_{\bm S}(w_n,\mathcal S)
\le
\kappa\,d_{\bm S^{-1}}(0,\bm{F}w_n)
\le
\kappa\,\|u_n^*\|_{\bm S^{-1}}
\le
\kappa c_0\,\|x_n-w_n\|_{\bm P}.
\end{equation*}
In addition, for $c_1=\sqrt{\lambda_{\max} (\bm S)/\lambda_{\min}(\bm P)}$, we have  that
\begin{align}
(\forall n\ge n_0 ) \qquad d_{\bm S}(x_n,\mathcal S)
&\le
\|x_n-w_n\|_{\bm S}+d_{\bm S}(w_n,\mathcal S)\nonumber\\
&\le
c_1\|x_n-w_n\|_{\bm P}
+
\kappa c_0\,\|x_n-w_n\|_{\bm P}\nonumber\\
&=
(c_1+\kappa c_0)\,\|x_n-w_n\|_{\bm P}. \label{eq:xw_lower}
\end{align}
% Equivalently,
% \begin{equation}\label{eq:xw_lower}
% \|x_n-w_n\|_{\bm P}
% \ge
% \frac{1}{c_1+\kappa c_0}\,d_{\bm S}(x_n,\mathcal S),
% \qquad \forall n\ge n_0.
% \end{equation}
%
Combining \eqref{eq:key_dn} and \eqref{eq:xw_lower}, we obtain
\begin{equation}\label{eq:dn_lower}
(\forall n\ge n_0 ) \qquad \|d_n\|_{\bm S}
\ge
\frac{\mu_\sigma}{c_1+\kappa c_0}\,d_{\bm S}(x_n,\mathcal S).
\end{equation}
On the other hand,  by \eqref{eq:keyInq}, we deduce
\begin{equation*}
 d_{\bm S}(x_{n+1},\mathcal S)^2
\le
d_{\bm S}(x_n,\mathcal S)^2
-
\lambda_n(2-\lambda_n)\|d_n\|_{\bm S}^2.
\end{equation*}
Therefore, in view of \eqref{eq:dn_lower}, we conclude
\begin{equation*}
(\forall n\ge n_0 ) \qquad d_{\bm S}(x_{n+1},\mathcal S)^2
\le
d_{\bm S}(x_n,\mathcal S)^2
-
\lambda_n(2-\lambda_n)
\left(\frac{\mu_\sigma}{c_1+\kappa c_0}\right)^2
d_{\bm S}(x_n,\mathcal S)^2.
\end{equation*}
Since $\lambda_n \in [\underline{\lambda},\overline{\lambda}] \subset ~]0,2[$, we have
\begin{equation*}
\inf_{n\in\N}\lambda_n(2-\lambda_n)
\left(\frac{\mu_\sigma}{c_1+\kappa c_0}\right)^2
 >0,
\end{equation*}
therefore, there exists $\tau \in ]0,1[$ such that
\begin{equation*}
(\forall n\ge n_0 ) \qquad d_{\bm S}(x_{n+1},\mathcal S)^2
\le
(1-\tau)\,d_{\bm S}(x_n,\mathcal S)^2.
\end{equation*}
This proves \eqref{eq:lin_rate_NFBE}. Finally, since Theorem~\ref{teo:convergencia} already states that $(x_n)$ converges weakly to some $\bar x\in\mathcal S$, the decay $
d_{\bm S}(x_n,\mathcal S)\to 0$ implies that $
x_n\to \bar x$ strongly. %The estimate \eqref{eq:lin_rate_NFBE} then yields the \(R\)-linear convergence of $(x_n)$ to $\bar x$.

\ref{thm:linear_NFBE2} If Assumption~\ref{asume:NFBEE} holds, for each $n \in \N$, $\bm{M}_n$ is Lipschitz continuous and strongly monotone with respect to ${\bm{S}}$ \cite[Proposition~2.1]{MorinBanertGiselsson2022}. Moreover, Algorithm~\ref{alg:NFBEE} is a particular instance of Algorithm~\ref{alg:NFBE} by Theorem~\ref{teo:NFBEE}. Hence, the proof is analogous to the proof of \ref{thm:linear_NFBE1}.
\end{proof}
\begin{rem}
   Note that,  when $\bm{A}+\bm{C}$ is strongly monotone,  \eqref{eq:lin_rate_NFBE_} holds for $n_0 = 0$ and some $r \in \RPP$ in view of Remark~\ref{rem:stronglymonotone}. In that case, the linear convergence is global.
\end{rem}
\subsubsection{Strong convergence via a Haugazeau-type projection}

We now introduce a strongly convergent variant of Algorithm~\ref{alg:NFBE}. The idea is to combine the separating half-space $\bm H_n$, generated by the inexact warped-resolvent step, with an additional half-space $\bm W_n$ of Haugazeau type.

\begin{algo}\label{alg:NFBE-H}
In the context of Problem~\ref{pro:main} and Assumption~\ref{asume:NFBE}, let $x_0\in\HH$ and consider the following recurrence:
\begin{equation}\label{eq:algWR-H}
(\forall n\in\N)\quad 
\begin{array}{l}
		\left\lfloor
		\begin{array}{l}
			%w_n = (\bm{M}_n+\bm{A}+\bm{C})^{-1}(\bm{M}_nx_n) \Leftrightarrow w_n^* \in \bm{A} w_n \\
            \textnormal{ find } (w_n,v_n) \in \gra\bm{A}
            \textnormal{ such that }\\
			\left\lfloor
		\begin{array}{l} w_n^* = \bm{M}_nx_n-\bm{M}_nw_n-\bm{C}w_n,\\
       e_n = v_n - w_n^*,\\        
           \|e_n\|_{\bm{P}^{-1}}\leq \sigma \|w_n-x_n\|_{\bm{P}},	
		\end{array}
		\right.\\[3mm]
t_n^*=v_n+\bm Cx_n,\\[1mm]
\bm H_n=
\left\{
z\in\HH:\ 
\scal{ z-w_n}{t_n^*}
\le
\dfrac{1}{4\beta}\|w_n-x_n\|_{\bm P}^2
\right\},\\[4mm]
\bm W_n=
\left\{
z\in\HH:\ 
\scal {z-x_n}{x_0-x_n}_{\bm S}\le 0
\right\},\\[4mm]
x_{n+1}=P_{\bm H_n\cap \bm W_n}^{\bm S}(x_0).
\end{array}
\right.
\end{array}
\end{equation}
\end{algo}

The half-space $\bm H_n$ contains the solution set by construction, while $\bm W_n$ is the standard Haugazeau half-space that forces the iterates toward the best approximation of $x_0$ from the solution set. We next show that Algorithm~\ref{alg:NFBE-H} is well-defined. The argument follows the same line as \cite[Proposition~3]{solodov2000forcing} adapted to our setting.

\begin{prop}\label{prop:well-defined}
In the context of Problem~\ref{pro:main} and Assumption~\ref{asume:NFBE}, define $
\mathcal S:=\zer(\bm A+\bm C)\neq\emptyset$. Then, for any $n \in \N$, if $x_{n} \in \HH$ is generated by Algorithm~\ref{alg:NFBE-H}, the following assertions hold
\begin{enumerate}
    \item\label{prop:well-defined1} $\mathcal S\subset \bm H_n\cap \bm W_n$;
    \item\label{prop:well-defined2} $x_{n+1}$ is well-defined;
    \item\label{prop:well-defined3} $\mathcal S\subset \bm W_{n+1}$.
\end{enumerate}
\end{prop}

\begin{proof} We proceed by induction. Suppose first that $n =0$, by Theorem~\ref{teo:convergencia}.\ref{h_k} we have $\mathcal S\subset \bm H_0$. On the other hand, by definition,
\begin{equation*}
\bm W_0 = \menge{z\in\HH}{\scal{z-x_0}{x_0-x_0}_{\bm S}\le 0}  =  \HH.
\end{equation*}
Hence, $\mathcal S\subset \bm H_0\cap \bm W_0$. Since $\mathcal S\neq\emptyset$, it follows that $\bm H_0\cap \bm W_0$ is nonempty. Moreover, it is closed and convex because both $\bm H_0$ and $\bm W_0$ are closed convex half-spaces. Therefore, $x_1=P_{\bm H_0\cap \bm W_0}^{\bm S}(x_0)$ is well-defined.  Moreover, by the characterization of the projection in the metric induced by $\bm S$, we have
\begin{equation*}
(\forall z\in \bm H_0\cap \bm W_0) \qquad   \scal {x_0-x_1}{z-x_1}_{\bm S}\le 0.
\end{equation*}
Since $\mathcal S\subset \bm H_0\cap \bm W_0$, the above inequality holds in particular for every
$z\in\mathcal S$. By the definition of $\bm W_1$, this means that $\mathcal S\subset \bm W_1$. Thus, \ref{prop:well-defined1}, \ref{prop:well-defined2}, and \ref{prop:well-defined3} hold for $n=0$.

Assume now that, $n\geq 1$ and  that
\begin{equation*}
\mathcal S\subset \bm H_n\cap \bm W_n,
\qquad
x_{n+1}=P_{\bm H_n\cap \bm W_n}^{\bm S}(x_0)
\quad\text{is well-defined, and}
\quad
\mathcal S\subset \bm W_{n+1}.
\end{equation*}
Suppose that Algorithm~\ref{alg:NFBE-H} does not stop at iteration $n+1$.
It follows from Theorem \ref{teo:convergencia}.\ref{h_k} that $\mathcal S\subset \bm H_{n+1}$, thus, in view of $\mathcal S\subset \bm W_{n+1}$, we deduce $\mathcal S\subset \bm H_{n+1}\cap \bm W_{n+1}$. Moreover, since $\mathcal S\neq\emptyset$ and $\bm H_{n+1}$ and $\bm W_{n+1}$ are both closed convex half-spaces, we conclude that $\bm H_{n+1}\cap \bm W_{n+1}$ is nonempty, closed and convex. Therefore, $x_{n+2}=P_{\bm H_{n+1}\cap \bm W_{n+1}}^{\bm S}(x_0)$ is well-defined. Using once more the characterization of the metric projection, for every $z \in \bm H_{n+1}\cap \bm W_{n+1}$, we have
\begin{equation*}
\scal{x_0-x_{n+2}}{z-x_{n+2}}_{\bm S}\le 0.
\end{equation*}
Since $\mathcal S\subset \bm H_{n+1}\cap \bm W_{n+1}$, the latter inequality holds for every
$z\in\mathcal S$, which is precisely the statement $\mathcal S\subset \bm W_{n+2}$. This completes the induction and the proof.
\end{proof}

\begin{teo}[Strong convergence of Algorithm~\ref{alg:NFBE-H}]
\label{teo:main-NFBE-H}
In the context of Problem~\ref{pro:main} and Assumption~\ref{asume:NFBE}.
Let $x_0\in\HH$, and let $(x_n)_{n\in\N}$ be generated by Algorithm~\ref{alg:NFBE-H}. Then the following hold:
\begin{enumerate}
    \item\label{teo:main-NFBE-H-i} For every $n\in\N$,
    \begin{equation}\label{eq:S}
        \|x_n-x_0\|_{\bm S}\le d_{\bm S}(x_0,\mathcal S).
    \end{equation}
    % $$
    % d_{\bm S}(x_0,\mathcal S):=\inf_{x\in\mathcal S}\|x-x_0\|_{\bm S}.
    % $$
    \item\label{teo:main-NFBE-H-ii} For every $n\in \N$,
    \begin{equation}\label{eq:haugazeau-descent}
        \|x_{n+1}-x_0\|_{\bm S}^2
        \ge
        \|x_n-x_0\|_{\bm S}^2
        +
        \|x_{n+1}-x_n\|_{\bm S}^2.
    \end{equation}
    In particular,
    \begin{equation*}
                \sum_{n\in\N}\|x_{n+1}-x_n\|_{\bm S}^2 < +\infty,
        \qquad
        x_{n+1}-x_n \to 0.
    \end{equation*}
    \item\label{teo:main-NFBE-H-iii}
    Assume, in addition, that one of the following conditions holds:
    \begin{enumerate}
        \item\label{teo:main-NFBE-H-iiia} $w_n-x_n \to 0$ and $\bm M_n w_n-\bm M_n x_n \to 0$;
        \item\label{teo:main-NFBE-H-iiib} $\bm M_n$ is $\alpha$-strongly monotone with respect to $\bm P$ for some $\alpha\in \Bigl]\frac{1}{4\beta}+\sigma,+\infty\Bigr[$ 
        and $\zeta$-Lipschitz continuous.
    \end{enumerate}
    Then $(x_n)_{n\in\N}$ converges weakly to a point in $\mathcal S$.
    \item\label{teo:main-NFBE-H-iv}
    Under the assumptions of \ref{teo:main-NFBE-H-iii}, the sequence $(x_n)_{n\in\N}$ converges strongly to
    $$
    x^\ast=P_{\mathcal S}^{\bm S}(x_0).
    $$
\end{enumerate}
\end{teo}
\begin{proof} \ref{teo:main-NFBE-H-i}
By Proposition~\ref{prop:well-defined}, for every $n \in \N$, 
$\mathcal S\subset \bm H_n\cap \bm W_n$ and $x_{n+1} = P_{\bm H_n\cap \bm W_n}^{\bm{S}}(x_0)$. Hence, for every $x\in\mathcal S$,
\begin{equation*}
\|x_{n+1}-x_0\|_{\bm S}\le \|x-x_0\|_{\bm S}.
\end{equation*}
Taking the infimum over $x\in\mathcal S$, we obtain \eqref{eq:S}.

\ref{teo:main-NFBE-H-ii} Since $x_{n+1}\in \bm W_n$, it follows from the definition of $\bm W_n$ that
\begin{equation*}
\scal {x_{n+1}-x_n}{x_0-x_n}_{\bm S}\le 0.
\end{equation*}
Therefore,
\begin{align*}
\|x_{n+1}-x_0\|_{\bm S}^2&=
\|x_{n+1}-x_n+x_n-x_0\|_{\bm S}^2\\
&=
\|x_{n+1}-x_n\|_{\bm S}^2
+\|x_n-x_0\|_{\bm S}^2
+2\scal {x_{n+1}-x_n}{x_n-x_0}_{\bm S}\\
&\ge
\|x_{n+1}-x_n\|_{\bm S}^2
+\|x_n-x_0\|_{\bm S}^2,
\end{align*}
which proves \eqref{eq:haugazeau-descent}. Summing \eqref{eq:haugazeau-descent} from $n=0$ to $N$ and using \eqref{eq:S}, we get
\begin{equation*}
\sum_{n=0}^N \|x_{n+1}-x_n\|_{\bm S}^2
\le
\|x_{N+1}-x_0\|_{\bm S}^2
\le
d_{\bm S}(x_0,\mathcal S)^2.
\end{equation*}
Hence,
\begin{equation*}
\sum_{n\in\N}\|x_{n+1}-x_n\|_{\bm S}^2<+\infty.
\end{equation*}
In particular,
\begin{equation*}
\|x_{n+1}-x_n\|_{\bm S}\to 0,
\end{equation*}
that is, $x_{n+1}-x_n\to 0$. 

\ref{teo:main-NFBE-H-iiia}. %If we assume 
% $$
% w_n-x_n\to 0
% \qquad\text{and}\qquad
% \bm M_nw_n-\bm M_nx_n\to 0.
% $$
The conclusion follows as in the proof of Theorem~\ref{teo:convergencia}.

\ref{teo:main-NFBE-H-iiib} Assume now that $\bm M_n$ is $\alpha$-strongly monotone and $\zeta$-Lipschitz continuous. By Lemma~\ref{lem:general_projection}, \eqref{eq:del_k}, and \eqref{eq:d_S}, for every $n \in \N$, we have
\begin{equation*}
\|P_{\bm H_n}(x_n)-x_n\|_{\bm S}
= \frac{\delta_n}{\|t_n^*\|_{\bm S^{-1}}}
\ge \mu_{\sigma}\|x_n-w_n\|_{\bm P}
\end{equation*}
where $\mu_\sigma$ is defined in \eqref{eq:defmusigma}.
%for some constant $\mu_\sigma>0$. 
Now, for every $n \in \N$, $x_{n+1}\in \bm H_n$, then
\begin{equation*}
(\forall n \in \N) \quad \|x_{n+1}-x_n\|_{\bm S}\ge \|P_{\bm H_n}(x_n)-x_n\|_{\bm S}.
\end{equation*}
Hence, by \ref{teo:main-NFBE-H-ii} and the Lipschitz continuity of $\bm{M}_n$ %and  \eqref{eq:algWR} (using equivalence norm), 
we have  
\begin{equation*}
    \lim_{n\to +\infty}\|x_n -w_n\|= \lim_{k\to + \infty}\|\bm{M}_n x_n - \bm{M}_n w_n\|=0,%\lim_{n\to +\infty}\|e_n\|=0.
\end{equation*}
which implies \ref{teo:main-NFBE-H-iiia} and the result follows.

\ref{teo:main-NFBE-H-iv} Finally, let $x^\ast=P_{\mathcal S}^{\bm S}(x_0)$ and let $\hat{x} \in \mathcal{S}$ be the weak limit point of $(x_n)_{n\in\N}$.  It follows from \ref{teo:main-NFBE-H-i} that
\begin{equation}\label{eq:main-NFBE-H-iv}
(\forall n \in \N) \quad \|x_n-x_0\|_{\bm S}\le d_{\bm S}(x_0,\mathcal S)= \|x^\ast-x_0\|_{\bm S}.
\end{equation}
By the weak lower semicontinuity of $\|\cdot\|_{\bm{S}}$ \cite[Lemma~2.42]{bauschkebook2017} we have
\begin{equation*}
\|\hat x-x_0\|_{\bm S}
\le
\liminf \|x_n-x_0\|_{\bm S}
\le
\|x^\ast-x_0\|_{\bm S}.
\end{equation*}
In addition, since $x^\ast=P_{\mathcal S}^{\bm S}(x_0)$, we have $\|x^\ast-x_0\|_{\bm S} \leq \|\hat x-x_0\|_{\bm S}$ and then $\hat x=x^\ast$. We conclude that $x_n\rightharpoonup x^\ast$. Therefore, $x_n\to x^\ast$ in view of \eqref{eq:main-NFBE-H-iv} and \cite[Lemma~2.51(i)]{bauschkebook2017}.
% Moreover,
% \begin{equation*}
% \|x_n-x_0\|_{\bm S}\to \|x^\ast-x_0\|_{\bm S}.
% \end{equation*}
% Therefore, by \cite[Corollary~2.52]{bauschkebook2017}.
\end{proof}

\section{Particular cases of the inexact warped algorithm}\label{sec:particular}
In this section, we present several applications of the explicit inexact warped resolvent framework. In particular, we derive inexact variants of well-known splitting methods, including FPDHF \cite{roldan2025forward}, Condat--V\~u \cite{Condat13}, Chambolle--Pock \cite{chambolle2016AN}, FBHF \cite{BricenoDavis2018}, and Tseng's splitting \cite{Tseng2000SIAM}. These methods arise as particular instances of our general scheme when applied to the following primal-dual monotone inclusion problem.
\begin{pro}\label{pro:mainPD}
	Let $(\H,\scal{\cdot}{\cdot})$ and $(\G,\scal{\cdot}{\cdot})$ 
	be real Hilbert spaces, let $A: \H \to 2^\H$ and $B:\G \to 
	2^\G$ be maximally 
	monotone operators,  let $L\colon \H \to \G$ be 
	a bounded linear operator, 
	let 
	$D:\H 
	\to \H$ be a $\zeta$-Lipschitz continuous and monotone operator for $\zeta \in 
	\RPP$, and let 
	$C:\H \to \H$ be a 
	$\beta$-cocoercive operator for $\beta \in \RPP$. The problem is to
	\begin{equation}\label{eq:mainPD}
		\text{find }  (x,u) \in \H \times \G \text{ such that } 
		\begin{cases}
			0 &\in (A+C+D)x+L^*u\\
			0 &\in B^{-1}u-Lx.
		\end{cases} 
	\end{equation}
	We assume that the solution set of this problem is nonempty.
\end{pro}
% Note that, given $(\hat{x},\hat{u}) \in \bm{Z}$, $\hat{x}$ is a 
% solution to the primal monotone inclusion
% \begin{equation}\label{eq:primalinclu}
% 	\text{find }  x \in \H \text{ such that } 
% 	0 \in (A+L^*BL+C+D)x
% \end{equation}
% and $\hat{u}$ is a solution to the dual monotone inclusion
% \begin{equation}\label{eq:dualinclu}
% 	\text{find }  u \in \G \text{ such that } 
% 	0 \in  B^{-1}u - L(A+C+D)^{-1}(-L^*u).
% \end{equation}
 Note that, this problem is a primal-dual formulation of \eqref{eq:probPDintro}.
Algorithms considering exact resolvent steps for solving this monotone inclusion have been studied, for example, in \cite{AttouchBricenoCombettes2010,CombettesMinh2022,Comb13,CombettesEckstein2018MP}. Inexact variants, allowing errors in the resolvent computations, have also been considered in \cite{Zong2018,rasch2020inexact,AlvesLorenzNaldi2026}. We first focus on the case when $B=0$ and $L=0$; thus, we study inexact versions of the FBHF, Tseng's splitting, and FB algorithms. Next, we address the case when $B\neq 0$ and $L\neq0$.
\subsection{FBHF, FBF, and FB with error}
In  this subsection, we propose versions of FBHF, FBF, and FB that allow for the inexact calculation of the backward steps.  We first introduce the following algorithm, which is a particular case of Algorithm~\ref{alg:NFBE}.
\begin{algo}\label{alg:FBHFEP} In the context of Problem~\ref{pro:mainPD}, let $x_0 \in \H$, $\gamma \in \RPP$, $\sigma \in [0,1[$, let $(\lambda_n)_{n \in \N}$ be a sequence in $[\underline{\lambda},\overline{\lambda}]\subset~]0,2[$, and consider the following recurrence.
    \begin{equation}\label{eq:algFBHFEP}
	(\forall n\in\N)\quad 
	\begin{array}{l}
		\left\lfloor
		\begin{array}{l}
            \textnormal{ find } (z_n,y_n) \in \gra A 
            \textnormal{ such that }\\
			\left\lfloor
		\begin{array}{l} z_n^* = x_n/\gamma - Cx_n-Dx_n-z_n /\gamma \\
       e_n = y_n - z_n^*,\\        
           \|e_n\|\leq \sigma \|z_n-x_n\|,	
		\end{array}
		\right.\\
        t_n^* = y_n+Dz_n+Cx_n\\
			\delta_n = \scal{x_n-z_n}{t_n^*}-\frac{1}{4\beta}\|z_n-x_n\|^2 \\
			d_n = \begin{cases}
				\dfrac{\delta_n}{\|t^*_n\|^2} t^*_n, &\textnormal{ if }
				\delta_n >0; \\
				0,  &\textnormal{ otherwise } 
			\end{cases}\\
			x_{n+1}=x_n-\lambda_n d_n.
		\end{array}
		\right.
	\end{array}
\end{equation}
\end{algo}
% \begin{algo}\label{alg:FBHFEP-H}In the context of Problem~\ref{pro:mainPD}, let $x _0 \in \H$, $\gamma \in \RPP$, $\sigma \in [0,1[$ and consider the following recurrence.
% \begin{equation}\label{eq:algFBHFEP-H}
% (\forall n\in\N)\quad 
% \begin{array}{l}
% 		\left\lfloor
% 		\begin{array}{l}
% 			%w_n = (\bm{M}_n+\bm{A}+\bm{C})^{-1}(\bm{M}_nx_n) \Leftrightarrow w_n^* \in \bm{A} w_n \\
%             \textnormal{ find } (z_n,y_n) \in \gra\bm{A}
%             \textnormal{ such that }\\
% 			\left\lfloor
% 		\begin{array}{l} z_n^* = x_n/\gamma-Cx_n-Dx_n-z_n/\gamma,\\
%        e_n = y_n - z_n^*,\\        
%            \|e_n\|\leq \sigma \|z_n-x_n\|,	
% 		\end{array}
% 		\right.\\[3mm]
% t_n^*=y_n+Dz_n+Cx_n,\\[1mm]
% \bm H_n=
% \left\{
% z\in\HH:\ 
% \langle z-z_n,t_n^*\rangle
% \le
% \dfrac{1}{4\beta}\|z_n-x_n\|^2
% \right\},\\[4mm]
% \bm W_n=
% \left\{
% z\in\HH:\ 
% \scal{z-x_n}{x_0-x_n} \le 0
% \right\},\\[4mm]
% x_{n+1}=P_{\bm H_n\cap \bm W_n}(x_0).
% \end{array}
% \right.
% \end{array}
% \end{equation}
% \end{algo}
\begin{teo}\label{teo:FBHFImplicit}
 In the context of Problem~\ref{pro:mainPD}, consider the sequence $(x_n)_{n \in \N}$ generated by Algorithm~\ref{alg:FBHFEP}. Suppose that
 \begin{equation}\label{eq:stepsizesFBHFI}
				1-\frac{5\gamma}{4\beta} -\gamma (\zeta+\sigma) >0.
                %\gamma \in  \left]0,\frac{4\beta}{1+\sqrt{1+16(\zeta+\sigma)^2\beta^2}}\right[.
			\end{equation}
Then, $(x_n)_{n \in \N}$ converges weakly to $x \in \zer(A+C+D)$.
 \end{teo}\begin{proof}
   Let $\gamma \in \R$ and define the operators
	\begin{equation}\label{eq:defAPD1}
		\begin{aligned}
		&\bm{A} \colon \H \to 2^\H \colon x \mapsto Ax+Dx,\\
&\bm{M} \colon \H \to \H \colon x \mapsto x/\gamma-Cx-Dx,\\
&\bm{C} \colon \H \to \H \colon x \mapsto Cx,\\
&\bm{S} \colon \H \to \H \colon x\mapsto x.
		\end{aligned}
	\end{equation}  Since $D$ is monotone and Lipschitz, $\bm{A}$ is maximally monotone \cite[Corollary~25.5]{bauschkebook2017}. Now, for every $n \in \N$,
define 
\begin{equation}\label{eq:defvar}
    \gamma_n = \gamma, \bm{M}_n = \bm{M},  w_n = z_n,  v_n = y_n+ Dw_n, \textnormal{ and } w_n^* = z_n^*+Dw_n. 
\end{equation}
% \begin{align*}
%     &w_n = z_n,\\
%     &v_n = y_n+Dw_n,\\
%     &w_n^* = z_n^*+Dw_n\\
%     &\bm{M}_n = \bm{M}\\
%     &\gamma_n = \gamma.
% \end{align*}
Hence, it follows from \eqref{eq:algFBHFEP} that, for every $n \in \N$, 
$v_n \in \bm{A}w_n$, $w_n^* = \bm{M}_nx_n - \bm{M}_n w_n -\bm{C}w_n$, $e_n = v_n - w_n^*$, and $t_n^* = v_n+\bm{C}x_n$.
% \begin{align*}
% &v_n \in \bm{A}w_n\\
% &w_n^* = \bm{M}_nx_n - \bm{M}_n w_n -\bm{C}w_n\\
% &e_n = v_n - w_n^*,\\        
% &\|e_n\|_{\bm{S}^{-1}}\leq \sigma \|w_n-x_n\|_{\bm{S}}\\
% &x_{n+1}=x_n-\gamma_n \bm{S}^{-1}(v_n+\bm{C}x_n),
% \end{align*}
Therefore, Algorithm~\ref{alg:FBHFEP} is a particular instance of Algorithm~\ref{alg:NFBE}. Moreover, since $C$ is $1/\beta$-Lipschitz and $D$ is $\zeta$-Lipschitz, we have that $\bm{M}$ is $(1/\gamma+1/\beta+\zeta)$-Lipschitz continuous, and for every $(x,y) \in \H^2$
\begin{align*}
    \gamma\scal{\bm Mx-\bm My}{x-y} &=\|x-y\|^2-\gamma\scal{Cx-Cy}{x-y}-\gamma\scal{Dx-Dy}{x-y}\\
    &\geq \left(1-\frac{\gamma}{\beta}-\gamma\zeta\right)\|x-y\|^2,
    % \scal{\bm  z-\bm R\widetilde{z}}{z-\widetilde{z}} -  \gamma\scal{\bm C z-\bm C\widetilde{z}}{z-\widetilde{z}}  + \|z-\widetilde{z}\|^2_{\bm{S}}\\
    % &\geq -\|\bm R z-\bm R\widetilde{z}\|_{\bm{S}^{-1}}\|\|z-\widetilde{z}\|_{\bm{S}}- \gamma\|\bm C z-\bm C\widetilde{z}\|_{\bm{S}^{-1}}\|z-\widetilde{z}\|_{\bm{S}} +   \|z-\widetilde{z}\|^2_{\bm{S}}\\
    % & \geq \left(1-\widehat{\zeta} - \frac{\gamma}{\widehat{\beta}}  \right)\|z-\widetilde{z}\|^2_{\bm{S}}.
\end{align*}
thus, $\bm{M}$ is $\alpha$-strongly monotone for $\alpha = 1/\gamma - 1/\beta -\zeta$. In view of \eqref{eq:stepsizesFBHFI} we have $\alpha \in ]1/(4\beta) + \sigma , +\infty[$ and the result follows by Theorem~\ref{teo:convergencia}.\ref{teo:convergencia2b}.
% Moreover, since $\zeta_n = \zeta\gamma$ for every $n \in \N$, \eqref{eq:stepsizes} reduces to  $1-\varepsilon-\gamma^2(\zeta+\sigma)^2 >0$  and $2\beta \varepsilon \geq \gamma$. Choosing $\varepsilon$ such that $(1-\varepsilon)/(\zeta+\sigma)^2=4\beta^2\varepsilon^2$, both inequalities hold for $\gamma$ as in \eqref{eq:stepsizesFBHF}. 
% The result follows from Theorem~\ref{teo:NFBEE}. 
\end{proof}
Now, we present an explicit version of Algorithm~\ref{alg:FBHFEP}, whose convergence follows directly from Theorem~\ref{teo:NFBEE}. This scheme can be interpreted as an inexact variant of the FBHF algorithm.
\begin{algo}\label{alg:FBHFE} In the context of Problem~\ref{pro:mainPD}, let $x_0 \in \H$, $\gamma \in \RPP$, $\sigma \in [0,1[$, and consider the following recurrence.
    \begin{equation}\label{eq:algFBHFE}
	(\forall n\in\N)\quad 
	\begin{array}{l}
		\left\lfloor
		\begin{array}{l}
            \textnormal{ find } (z_n,y_n) \in \gra A 
            \textnormal{ such that }\\
			\left\lfloor
		\begin{array}{l} z_n^* = x_n/\gamma - Cx_n-Dx_n-z_n /\gamma \\
       e_n = y_n - z_n^*,\\        
           \|e_n\|\leq \sigma \|z_n-x_n\|,	
		\end{array}
		\right.\\
			x_{n+1}=z_n+\gamma (Dx_n-Dz_n-e_n).
		\end{array}
		\right.
	\end{array}
\end{equation}
\end{algo}
Note that, when $\sigma=0$, and hence $e_n=0$, \eqref{eq:algFBHFE} coincides with \cite[Eq. (2.12)]{BricenoDavis2018} for $X = \H$. The following result establishes the convergence of Algorithm~\ref{alg:FBHFE}.
\begin{teo}
 In the context of Problem~\ref{pro:mainPD}, consider the sequence $(x_n)_{n \in \N}$ generated by Algorithm~\ref{alg:FBHFE}. Suppose that
 \begin{equation}\label{eq:stepsizesFBHF}
				\gamma \in  \left]0,\frac{4\beta}{1+\sqrt{1+16(\zeta+\sigma)^2\beta^2}}\right[.
			\end{equation}
Then, $(x_n)_{n \in \N}$ converges weakly to $x \in \zer(A+C+D)$.
 \end{teo}
 \begin{proof} Considering the operators defined in \eqref{eq:defAPD1}, we have that $\gamma\bm{M}-\gamma\bm{C}-\bm{S}=-\gamma D$ is $(\gamma \zeta)$-Lipschitz continuous.   
%Let $\gamma \in \R$ and define the operators
% 	\begin{equation}\label{eq:defAPD1}
% 		\begin{aligned}
% 		&\bm{A} \colon \H \to 2^\H \colon x \mapsto Ax+Dx,\\
% &\bm{M} \colon \H \to \H \colon x \mapsto x/\gamma-Cx-Dx,\\
% &\bm{C} \colon \H \to \H \colon x \mapsto Cx,\\
% &\bm{S} \colon \H \to \H \colon x\mapsto x.
% 		\end{aligned}
% 	\end{equation}    
% It follows that $\bm{A}$ is maximally monotone and that $\gamma\bm{M}-\gamma\bm{C}-\bm{S}=-\gamma D$ is $(\gamma \zeta)$-Lipschitz continuous. Now, for every $n \in \N$,
% define $\gamma_n = \gamma$, $\bm{M}_n = \bm{M}$,  $w_n = z_n$,  $v_n = y_n+ Dw_n$, and $w_n^* = z_n^*+Dw_n$. 
% \begin{align*}
%     &w_n = z_n,\\
%     &v_n = y_n+Dw_n,\\
%     &w_n^* = z_n^*+Dw_n\\
%     &\bm{M}_n = \bm{M}\\
%     &\gamma_n = \gamma.
% \end{align*}
Moreover, defining $\gamma_n$, $\bm{M}_n$, $w_n$, $v_n$ and $w_n^*$ as in \eqref{eq:defvar}, we have that $x_{n+1}=x_n-\gamma_n \bm{S}^{-1}(v_n+\bm{C}x_n)$, for every $n \in \N$.
Hence, Algorithm~\ref{alg:FBHFE} is a particular instance of Algorithm~\ref{alg:NFBEE}.
% \begin{align*}
% &v_n \in \bm{A}w_n\\
% &w_n^* = \bm{M}_nx_n - \bm{M}_n w_n -\bm{C}w_n\\
% &e_n = v_n - w_n^*,\\        
% &\|e_n\|_{\bm{S}^{-1}}\leq \sigma \|w_n-x_n\|_{\bm{S}}\\
% &x_{n+1}=x_n-\gamma_n \bm{S}^{-1}(v_n+\bm{C}x_n),
% \end{align*}
Moreover, since $\zeta_n = \zeta\gamma$ for every $n \in \N$, \eqref{eq:stepsizes} reduces to  $1-\varepsilon-\gamma^2(\zeta+\sigma)^2 >0$  and $2\beta \varepsilon \geq \gamma$. Choosing $\varepsilon$ such that $(1-\varepsilon)/(\zeta+\sigma)^2=4\beta^2\varepsilon^2$, both inequalities hold for $\gamma$ as in \eqref{eq:stepsizesFBHF}. 
The result follows from Theorem~\ref{teo:NFBEE}. 
\end{proof}
\begin{rem} \begin{enumerate}
    \item In the case where $C = 0 $, Algorithm~\ref{alg:FBHFE} reduces to 
        \begin{equation}\label{eq:algFBFE}
	(\forall n\in\N)\quad 
	\begin{array}{l}
		\left\lfloor
		\begin{array}{l}
            \textnormal{ find } (z_n,y_n) \in \gra A 
            \textnormal{ such that }\\
			\left\lfloor
		\begin{array}{l} z_n^* = x_n/\gamma -Dx_n-z_n /\gamma \\
       e_n = y_n - z_n^*,\\        
           \|e_n\|\leq \sigma \|z_n-x_n\|,	
		\end{array}
		\right.\\
			x_{n+1}=z_n+\gamma (Dx_n-Dz_n-e_n),
		\end{array}
		\right.
	\end{array}
\end{equation}
which is an inexact resolvent version of the FBF algorithm. In this case, by taking $\beta \to +\infty$, the convergence of the algorithm is guaranteed for $\gamma \in  \left]0,\frac{1}{\zeta+\sigma}\right[$.
    \item When $D = 0 $, Algorithm~\ref{alg:FBHFE} reduces to 
        \begin{equation}\label{eq:algFBE}
	(\forall n\in\N)\quad 
	\begin{array}{l}
		\left\lfloor
		\begin{array}{l}
            \textnormal{ find } (z_n,y_n) \in \gra A 
            \textnormal{ such that }\\
			\left\lfloor
		\begin{array}{l} z_n^* = x_n/\gamma-Cx_n-z_n /\gamma \\
       e_n = y_n - z_n^*,\\        
           \|e_n\|\leq \sigma \|z_n-x_n\|,	
		\end{array}
		\right.\\
			x_{n+1}=z_n-\gamma e_n.
		\end{array}
		\right.
	\end{array}
\end{equation}
This recurrence is an inexact resolvent version of the FB algorithm and its convergence is guaranteed for $\gamma \in  \left]0,\frac{4\beta}{1+\sqrt{1+16\sigma^2\beta^2}}\right[$.
\end{enumerate}
\end{rem}
\subsection{FPDHF, Condat--V\~u, and Chambolle--Pock with error} 
In this subsection, we introduce an inexact variant of the FPDHF scheme, in which the resolvent evaluations in the backward steps are computed approximately. This leads naturally to inexact versions of the Condat--V\~u and Chambolle--Pock algorithms. 

First, we establish the following result, which identifies the operator properties required to apply the proposed inexact warped resolvent framework to the primal-dual problem \ref{pro:mainPD}.

\begin{prop}\label{prop:operators} In the context of Problem~\ref{pro:mainPD}, let $(\gamma,\tau) \in \RPP^2$, $\sigma \in [0,1[$, set $\HH = \H\times  \G$, and define the operators:
\begin{equation}\label{eq:defAPD}
\begin{aligned}
\bm A \colon \HH &\to 2^{\HH} :
(x,u)\mapsto (Ax+Dx+L^*u)\times(B^{-1}u-Lx),\\
\bm M \colon \HH &\to \HH :
(x,u)\mapsto \left(\frac{x}{\gamma}-Cx-Dx-L^*u,\,-Lx+\gamma LDx+\frac{u}{\tau}\right),\\
\bm C \colon \HH &\to \HH :
(x,u)\mapsto (Cx,0),\\
\bm S \colon \HH &\to \HH :
(x,u)\mapsto \left(x-\gamma L^*u,\,-\gamma Lx+\frac{\gamma}{\tau}u\right),\\
\bm T \colon \HH &\to \HH :
(x,u)\mapsto \left(\frac{x}{\gamma}-Dx,\frac{u}{\gamma}\right),\\
 \bm R \colon \HH &\to \HH :
(x,u)\mapsto \gamma (\bm M+\bm C)-\bm S.
\end{aligned}
\end{equation}
Then, the following assertions hold.
\begin{enumerate}
    \item\label{prop:operators1} $\bm{A}$ is maximally monotone.
    \item\label{prop:operators2} If $1-\gamma \tau \|L\|^2>0$,  $\bm{S}$ is a strongly monotone self-adjoint bounded linear operator and 
        \begin{equation}\label{eq:Sm1}
        \bm{S}^{-1}  \colon \HH \to \HH \colon (z,v) \mapsto \big((\id-\gamma\T L^*L)^{-1}(z+\tau L^*v),\tau (\id-\gamma\T LL^*)^{-1}(Lz+v/\gamma)\big).
    \end{equation}
\item\label{prop:operators3} $\bm{M}+\bm{C}=\bm{S}\circ \bm{T}$.
\item\label{prop:operators4}  $\bm C$ is $\widehat{\beta}$-cocoercive  with respect to $\bm S$ for $\widehat{\beta}=\beta(1-\tau\gamma\|L\|^2)$.
\item\label{prop:operators5} $\bm R$ is $\widehat{\zeta}$-Lipschitz with respect to $\bm S$ for $\widehat{\zeta}=\gamma\zeta(1-\tau\gamma\|L\|^2)^{-1/2}$.
\item\label{prop:operators6} $\bm M$ is $ \alpha$-strongly monotone with respect to $\bm S$ for $\alpha = \frac{1-\widehat{\zeta}}{\gamma}-\frac{1}{\widehat{\beta}}$.
\end{enumerate}
\end{prop}
\begin{proof}
\ref{prop:operators1} Since $D$ is monotone and Lipschitz, $A+D$ is maximally monotone \cite[Corollary~25.5]{bauschkebook2017}. The result follows by \cite[Proposition~2.7(iii)]{briceno2011SIAM}.

\ref{prop:operators2} This follows from \cite[Lemma~6.1]{MorinBanertGiselsson2022}.

\ref{prop:operators3} Direct.

\ref{prop:operators4} See \cite[Corollary~6.1]{MorinBanertGiselsson2022}.

\ref{prop:operators5} Let $z=(x,u) \in \HH$. We have $ \bm Rz
=
%\gamma(\bm M+\bm C)z-\bm S z
%=
(-\gamma Dx,\gamma^2LDx)=\bm{S}(-\gamma Dx,0)$. 
Hence, for $\tilde z=(\tilde x,\tilde u) \in \HH$, it follows from the $\zeta$-Lipschitz continuity of $D$ and \cite[Lemma~6.1]{MorinBanertGiselsson2022} that
\begin{align*}
\|\bm Rz-\bm R\tilde z\|_{\bm S^{-1}}
&= \|(-\gamma (Dx-D\tilde{x}),0)\|_{\bm{S}}\\
&=
\gamma\|Dx-D\tilde x\|\\
&\leq \gamma \zeta\|x-\tilde x\|\\
&\leq \frac{\gamma\zeta}{\sqrt{1-\tau\gamma\|L\|^2}} \|(x-\tilde x,u-\tilde u)\|_{\bm{S}}\\
&= \frac{\gamma\zeta}{\sqrt{1-\tau\gamma\|L\|^2}} \|z-\tilde z\|_{\bm{S}}
\end{align*}
and the result follows.

\ref{prop:operators6} Let $(z,\widetilde{z}) \in \HH^2$. It follows from \ref{prop:operators4} and \ref{prop:operators5} that 
\begin{align*}
    \gamma\scal{\bm Mz-\bm M\widetilde{z}}{z-\widetilde{z}} &= \scal{\bm R z-\bm R\widetilde{z}}{z-\widetilde{z}} -  \gamma\scal{\bm C z-\bm C\widetilde{z}}{z-\widetilde{z}}  + \|z-\widetilde{z}\|^2_{\bm{S}}\\
    &\geq -\|\bm R z-\bm R\widetilde{z}\|_{\bm{S}^{-1}}\|z-\widetilde{z}\|_{\bm{S}}- \gamma\|\bm C z-\bm C\widetilde{z}\|_{\bm{S}^{-1}}\|z-\widetilde{z}\|_{\bm{S}} +   \|z-\widetilde{z}\|^2_{\bm{S}}\\
    & \geq \left(1-\widehat{\zeta} - \frac{\gamma}{\widehat{\beta}}  \right)\|z-\widetilde{z}\|^2_{\bm{S}}.
\end{align*}
The result follows.
\end{proof}
% \textcolor{red}{\begin{equation*}
%     \frac{1}{\gamma}\left(1-\widehat{\zeta} - \frac{\gamma}{\widehat{\beta}}  \right) > \frac{1}{4\widehat{\beta}} + \sigma \Leftrightarrow 1-\frac{5\gamma}{4\widehat\beta}-(\widehat{\zeta}+\gamma \sigma) >0
% \end{equation*}
% parece ser más restrictiva que \eqref{eq:stepsizes} tomando $\varepsilon = \frac{\gamma}{2\beta}$.
% }
The following algorithm is a primal-dual version of Algorithm~\ref{alg:NFBE}.
\begin{algo}\label{alg:FPDHF} In the context of Problem~\ref{pro:mainPD}, let $(x_0,u_0) \in \H\times\G$, $(\gamma,\tau) \in \RPP^2$, $\sigma \in [0,1[$, let $(\lambda_n)_{n \in \N}$ be a sequence in $[\underline{\lambda},\overline{\lambda}]\subset~]0,2[$, and consider the following recurrence.
    \begin{equation}\label{eq:algFPDHF}
	(\forall n\in\N)\quad 
	\begin{array}{l}
		\left\lfloor
		\begin{array}{l}
            \textnormal{ find } (p_n,y_n) \in \gra A 
           \textnormal{ and } (q_n,v_n) \in \gra B^{-1}  \textnormal{ such that }\\
			\left\lfloor
		\begin{array}{l} p_n^{*} = x_n/\gamma - Cx_n-Dx_n-L^*u_n-p_n /\gamma \\
        s_n = p_n+\gamma (Dx_n-Dp_n) \\
        q_n^{*} =L(p_n+s_n-x_n)-q_n/\tau+u_n/\tau\\
        (e_n^1,e_n^2) = (y_n-p_n^{*},v_n-q_n^{*})\\
           \|(e_n^1,e_n^2)\|_{\bm{S}^{-1}}\leq \sigma \|(p_n-x_n,q_n-u_n)\|_{\bm{S}},	
		\end{array}
		\right.\\
(t_n^1,t_n^2) = (y_n+Dp_n+L^*q_n+Cx_n , v_n-Lp_n)\\
			\delta_n = \scal{x_n-p_n}{t_n^1}+\scal{u_n-q_n}{t_n^2}-\frac{1}{4 \widehat{\beta}}\|(x_n-p_n,u_n-q_n)\|^2_{{\bm{S}}} \\
			d_n = \begin{cases}
				\dfrac{\delta_n}{\|(t_n^1,t_n^2)\|^2_{\bm{S}^{-1}}}  \bm{S}^{-1}(t_n^1,t_n^2), &\textnormal{ if }
				\delta_n >0; \\
				0,  &\textnormal{ otherwise } 
			\end{cases}\\
			(x_{n+1},u_{n+1})=(x_n,u_n)-\lambda_n d_n
            \end{array}
		\right.
	\end{array}
\end{equation}
\end{algo}
\begin{teo}\label{thm:FPDHF-conv}
In the context of Problem~\ref{pro:mainPD}, let $\bigl((x_n,u_n)\bigr)_{n\in\N}$ be generated by Algorithm~\ref{alg:FPDHF}. Set $\widehat{\beta}=\beta(1-\tau\gamma\|L\|^2)$, $\widehat{\zeta}=\gamma \zeta (1-\tau\gamma\|L\|^2)^{-1/2}$, and assume that
\begin{equation}\label{eq:step-FPDHF}
1-\frac{5\gamma}{4\widehat\beta}-(\widehat{\zeta}+\gamma \sigma)>0.
%1-\tau\gamma\|L\|^2>0
\end{equation}
%and that there exists $\varepsilon\in(0,1)$ such that
%\begin{equation}\label{eq:cond-FPDHF}
%1-\varepsilon - \gamma^2 \left(\frac{\zeta}{\sqrt{1-\tau\gamma\|L\|^2}}
%+\sigma\right)^2 >0,
%\qquad
%2\beta\varepsilon(1-\tau\gamma\|L\|^2)\ge \gamma.
%\end{equation}
Then the sequence $\bigl((x_n,u_n)\bigr)_{n\in\N}$ converges weakly to a solution of Problem~\ref{pro:mainPD}.
\end{teo}
\begin{proof}
    Set $\HH=\H\times\G$ and consider the operators $\bm{A}$, $\bm{M}$, $\bm{C}$, and $\bm{S}$ defined in \eqref{eq:defAPD}. Note that \eqref{eq:step-FPDHF} implies that $1-\gamma \tau \|L\|^2>0$.
%{\color{red}Note that $\bm M + \bm C = \bm S \circ \bm T$.}
For every $n\in\N$, set
\begin{align}\label{eq:defvarPD}
  &\gamma_n=\gamma,\quad \bm M_n=\bm M,\quad \x_n=(x_n,u_n),\quad \w_n=(p_n,q_n),  \quad \bm v_n=(y_n+Dp_n+L^*q_n,\;v_n-Lp_n),\nonumber\\
&\w_n^*=(p_n^*+Dp_n+L^*q_n,\;q_n^*-Lp_n),
\quad \bm{e}_n=(e_n^1,e_n^2), \quad \textnormal{and} \quad 
 \bm{t}_n^*=(t_n^1,t_n^2).
\end{align}
Since $(p_n,y_n)\in\gra A$ and $(q_n,v_n)\in\gra B^{-1}$, we have
\begin{equation*}
    \bm v_n\in \bm A\w_n.
\end{equation*}
Moreover, by construction,
\begin{equation*}
  \w_n^*=\bm M_n\x_n-\bm M_n\w_n-\bm C\w_n,
\qquad \textnormal{and} \qquad 
\bm e_n=\bm v_n-\w_n^*.  
\end{equation*}
Therefore, Algorithm~\ref{alg:FPDHF} can be written equivalently as
\begin{equation}
	(\forall n\in\N)\quad 
	\begin{array}{l}
		\left\lfloor
		\begin{array}{l}
			%w_n = (\bm{M}_n+\bm{A}+\bm{C})^{-1}(\bm{M}_nx_n) \Leftrightarrow w_n^* \in \bm{A} w_n \\
            \textnormal{find } (\w_n,\bm{v}_n) \in \gra \bm{A}  \textnormal{ such that }\\
			\left\lfloor
		\begin{array}{l} \w_n^* = \bm{M}_n\x_n-\bm{M}_n\w_n-\bm{C}\w_n\\
			\bm{e}_n = \bm{v}_n - w_n^*\\
    |\bm{e}_n\|_{\bm{S}^{-1}}\leq \sigma \|\w_n-\x_n\|_{\bm{S}}
            		\end{array}
		\right.\\
			\bm{t}_n^* = \bm{v}_n+\bm{C}\x_n\\
			\delta_n = \scal{\x_n-\w_n}{\bm{t}_n^*}-\frac{1}{4\widehat{\beta}}\|\w_n-\x_n\|^2_{{\bm{S}}} \\
			d_n = \begin{cases}
				\dfrac{\delta_n}{\|\bm{t}^*_n\|^2_{\bm{S}^{-1}}}  \bm{S}^{-1}\bm{t}^*_n, &\textnormal{ if }
				\delta_n >0; \\
				0,  &\textnormal{ otherwise } 
			\end{cases}\\
			\x_{n+1}=\x_n-\lambda_n d_n,
		\end{array}
		\right.
	\end{array} \notag
\end{equation}
which is a particular instance of Algorithm~\ref{alg:NFBE}. In addition, by Proposition~\ref{prop:operators}, $\bm{A}$ is maximally monotone, 
$\bm M_n$ is $\alpha$-strongly monotone and $\widehat{\zeta}$-Lipschitz, with $\alpha = \frac{1-\widehat{\zeta}}{\gamma}-\frac{1}{\widehat{\beta}}$,  and $\bm{C}$ is $\widehat{\beta}$-cocoercive. Furthermore, in view of \eqref{eq:step-FPDHF} we have $\alpha \in ~]1/(4\widehat\beta)+\sigma,+ \infty[$. Finally, since $\zer(\bm A+\bm C)$ coincides with the solution set of Problem~\ref{pro:mainPD}, the result follows by Theorem~\ref{teo:convergencia}.\ref{teo:convergencia2b}.
\end{proof}

\begin{rem}\label{rem:FPDHFunaresolvente}
\begin{enumerate}
    \item \label{rmk1} Fix $\kappa\in~]0,1[$ and let
    \begin{equation} \label{eq:gamma}
        \gamma \in \left]0,\frac{1 - \zeta \sqrt{\kappa}}{\frac{5}{\beta\kappa}+\sigma}\right[.
    \end{equation}
Now, setting $\tau = (1-\kappa)/(\gamma \|L\|^2)$, we have $\kappa = 1 - \gamma \tau \|L\|^2$ and
\begin{equation*}
    \gamma < \frac{1- \zeta \sqrt{\kappa}}{\frac{5}{\beta\kappa}+\sigma} \Leftrightarrow 1-\frac{5\gamma}{4\beta\kappa}-(\zeta \sqrt{\kappa} +\gamma \sigma)>0 \Leftrightarrow  1-\frac{5\gamma}{4\widehat{\beta}}-(\widehat{\zeta} +\gamma \sigma)>0.
\end{equation*}
        %Therefore, if \eqref{eq:gamma} holds
    %
    %For any  $\gamma \in \RPP$ we can find $\tau \in \RPP$ such that $\kappa = 1 - \gamma \tau \|L\|^2$, that is, $\tau = (1-\kappa)/(\gamma \|L\|^2)$. In that case,  $\widehat{\beta} = \beta \kappa$, $\widehat{\zeta}=\zeta \sqrt{\kappa}$. 
    Therefore, a simple condition guaranteeing that \eqref{eq:step-FPDHF} holds is \eqref{eq:gamma} and $\tau = (1-\kappa)/(\gamma \|L\|^2)$.
\item\label{rmk2} If we consider approximations only on the resolvent of $A$, it is possible to avoid the calculations of $\bm{S}^{-1}$ in Algorithm~\ref{alg:FPDHF}. Indeed, if we consider $v_n = (u_n-q_n)/\tau+L(p_n-\gamma(y_n+Cx_n+Dp_n+L^*u_n))$ and we set $(a_n,b_n) = (y_n+Dp_n+Cx_n + L^*u_n, (u_n-q_n)/\gamma)$, after simple calculations, we deduce
\begin{align*}
&q_n = J_{\tau B^{-1}} \left( u_n +\tau L(s_n+p_n-x_n-\gamma (y_n-p_n^*))\right),\\
&e_n^2 %&= v_n-q_n^*\\
        % &= \frac{u_n-q_n}{\tau}+L(p_n-\gamma(y_n+Cx_n+Dp_n+L^*u_n)) - L(p_n+s_n-x_n)+\frac{q_n-u_n}{\tau}\\
         % &= -\gamma L \left( y_n -\left(\frac{x_n-p_n}{\gamma} -Cx_n-Dx_n-L^*u_n\right)\right)   \\
          = -\gamma L \left( y_n -p_n^*\right),\\
 & (e_n^1,e_n^2)=\bm{S}(y_n-p_n^*,0)       \\ 
 &   (t_n^1,t_n^2) = \bm{S}(a_n, b_n).
\end{align*}
% we have
% \begin{align*}
%     v_n \in B^{-1} q_n &\Leftrightarrow \frac{u_n-q_n}{\tau}+L(p_n-\gamma(y_n+Cx_n+Dp_n+L^*u_n)) \in B^{-1} q_n\\
%     &\Leftrightarrow q_n = J_{\tau B^{-1}} \left( u_n +\tau L(p_n-\gamma(y_n+Cx_n+Dp_n+L^*u_n))\right)\\
%     &\Leftrightarrow  q_n = J_{\tau B^{-1}} \left( u_n +\tau L(s_n+p_n-x_n-\gamma (y_n-p_n^*))\right)
% \end{align*}
    % \begin{align*}
    % e_n^2 &= v_n-q_n^*\\
    %      &= \frac{u_n-q_n}{\tau}+L(p_n-\gamma(y_n+Cx_n+Dp_n+L^*u_n)) - L(p_n+s_n-x_n)+\frac{q_n-u_n}{\tau}\\
    %       &= -\gamma L \left( y_n -\left(\frac{x_n-p_n}{\gamma} -Cx_n-Dx_n-L^*u_n\right)\right)   \\
    %       &= -\gamma L \left( y_n -p_n^*\right).
    % \end{align*}
Therefore, Algorithm~\ref{alg:FPDHF} reduces to the following routine
    \begin{equation}\label{eq:algFPDHFEE}
	(\forall n\in\N)\quad 
	\begin{array}{l}
		\left\lfloor
		\begin{array}{l}
            \textnormal{ find } (p_n,y_n) \in \gra A \textnormal{ such that }\\
			\left\lfloor
		\begin{array}{l} p_n^{*} = x_n/\gamma - Cx_n-Dx_n-L^*u_n-p_n /\gamma \\
        s_n = p_n+\gamma (Dx_n-Dp_n) \\
        q_n = J_{\tau B^{-1}} \left( u_n +\tau L(s_n+p_n-x_n-\gamma (y_n-p_n^*))\right)\\
           \|y_n-p_n^*\|\leq \sigma \|(p_n-x_n,q_n-u_n)\|_{\bm{S}},	
		\end{array}
		\right.\\
%(t_n^1,t_n^2) = (y_n+Dp_n+L^*q_n+Cx_n , v_n-Lp_n)\\
%(t_n^1,t_n^2) = \bm{S}(y_n+Dp_n+Cx_n + L^*u_n, (u_n-q_n)/\gamma)\\
(a_n,b_n) = (y_n+Dp_n+Cx_n + L^*u_n, (u_n-q_n)/\gamma)\\
(t_n^1,t_n^2) = \bm{S}(a_n, b_n)\\
			\delta_n = \scal{x_n-p_n}{t_n^1}+\scal{u_n-q_n}{t_n^2}-\frac{1}{4 \widehat{\beta}}\|(x_n-p_n,u_n-q_n)\|^2_{{\bm{S}}} \\
			%d_n = \begin{cases}
			%	\dfrac{\delta_n}{\|(t_n^1,t_n^2)\|^2_{\bm{S}^{-1}}}  \bm{S}^{-1}(t_n^1,t_n^2), &\textnormal{ if }
			%	\delta_n >0; \\
			%	0,  &\textnormal{ otherwise } 
			%\end{cases}\\
            d_n = \begin{cases}
				\dfrac{\delta_n}{\|(a_n,b_n)\|^2_{\bm{S}}}  (a_n,b_n), &\textnormal{ if }
				\delta_n >0; \\
				0,  &\textnormal{ otherwise } 
			\end{cases}\\
			(x_{n+1},u_{n+1})=(x_n,u_n)-\lambda_n d_n.
            \end{array}
		\right.
	\end{array}
\end{equation}
\end{enumerate}
\end{rem}
Now, by applying Algorithm~\ref{alg:NFBEE} to the operators defined in \eqref{eq:defAPD}, we derive the following explicit scheme for solving Problem~\ref{pro:mainPD}.
\begin{algo}\label{alg:FPDHFE} In the context of Problem~\ref{pro:mainPD}, let $(x_0,u_0) \in \H\times\G$, $(\gamma,\tau) \in \RPP^2$, $\sigma \in [0,1[$, and consider the following recurrence.
    \begin{equation}\label{eq:algFPDHFE}
	(\forall n\in\N)\quad 
	\begin{array}{l}
		\left\lfloor
		\begin{array}{l}
            \textnormal{ find } (p_n,y_n) \in \gra A 
           \textnormal{ and } (q_n,v_n) \in \gra B^{-1}  \textnormal{ such that }\\
			\left\lfloor
		\begin{array}{l} p_n^{*} = x_n/\gamma - Cx_n-Dx_n-L^*u_n-p_n /\gamma \\
        s_n = p_n+\gamma (Dx_n-Dp_n) \\
        q_n^{*} =L(p_n+s_n-x_n)-q_n/\tau+u_n/\tau\\
        (e_n^1,e_n^2) = (y_n-p_n^{*},v_n-q_n^{*})\\
           \|(e_n^1,e_n^2)\|_{\bm{S}^{-1}}\leq \sigma \|(p_n-x_n,q_n-u_n)\|_{\bm{S}},	
		\end{array}
		\right.\\
			(x_{n+1},u_{n+1})=(s_n,q_n)- \gamma \bm{S}^{-1}(e_n^1,e_n^2).
		\end{array}
		\right.
	\end{array}
\end{equation}
\end{algo}
%
%Algorithm~\ref{alg:FPDHFE} can be viewed as a particular instance of Algorithm~\ref{alg:NFBE} applied to the primal--dual inclusion associated with Problem~\ref{pro:mainPD}. As a consequence, the weak convergence of the sequence generated by Algorithm~\ref{alg:FPDHFE} follows directly from the abstract convergence result established for our main scheme.
The next theorem establishes the weak convergence of Algorithm~\ref{alg:FPDHFE}.
\begin{teo}\label{thm:FPDHFE}
In the context of Problem~\ref{pro:mainPD},  let $\sigma\in[0,1)$, and let $\bigl((x_n,u_n)\bigr)_{n\in\N}$ be the sequence generated by Algorithm~\ref{alg:FPDHFE}. Suppose that
\begin{equation}\label{eq:stepsizesFPDHFE}
1-\varepsilon-\gamma^2\left(\frac{\zeta}{\sqrt{\,1-\tau\gamma\|L\|^2\,}}+\sigma\right)^2>0
\quad\text{and}\quad
2\beta\varepsilon(1-\tau\gamma\|L\|^2)\ge \gamma.
\end{equation}
Then $\bigl((x_n,u_n)\bigr)_{n\in\N}$ converges weakly to a solution of Problem~\ref{pro:mainPD}.
\end{teo}
\begin{proof} Consider the operators defined in \eqref{eq:defAPD} and the variables defined in \eqref{eq:defvarPD}. Similarly to the proof of Theorem~\ref{thm:FPDHF-conv} we have   $\w_n^*=\bm M_n\x_n-\bm M_n\w_n-\bm C\w_n$ and $ \bm e_n=\bm v_n-\w_n^*$. In addition, noting that, $\bm{M}+\bm{C} = \bm{S} \circ \bm{T}$ (Proposition~\ref{prop:operators}.\ref{prop:operators3}), we have $$\bm{T}\x_n-\bm{T}\w_n = \bm{S}^{-1}((\bm{M}_n+\bm{C})\x_n-(\bm{M}_n+\bm{C})\w_n).$$
Therefore,
\begin{align*}
(\forall n \in \N) \quad \x_{n+1}
&=(x_{n+1},u_{n+1})\\
&=(s_n,q_n)-\gamma \bm S^{-1}(e_n^1,e_n^2)\\
&=(p_n+\gamma (Dx_n-Dp_n,q_n)-\gamma \bm S^{-1}(e_n^1,e_n^2)\\
&=(x_n,u_n)-\gamma\left(\frac{x_n}{\gamma}-Dx_n,\frac{u_n}{\gamma}\right)+\gamma\left(\frac{p_n}{\gamma}-Dp_n,\frac{q_n}{\gamma}\right)-\gamma \bm S^{-1}(e_n^1,e_n^2)\\
&=\x_n-\gamma(\bm T\x_n-\bm T\w_n)-\gamma \bm S^{-1}\bm e_n\\
&=\x_n-\gamma \bm S^{-1}\bigl(\bm M_n\x_n+\bm C\x_n-\bm M_n\w_n-\bm C\w_n+\bm e_n\bigr)\\
&=\x_n-\gamma \bm S^{-1}\bigl(\w_n^*+\bm C\x_n+\bm e_n\bigr)\\
&=\x_n-\gamma \bm S^{-1}\bigl(\bm v_n+\bm C\x_n\bigr),
\end{align*}
thus, $(\x_n)_{n\in\N}$ is generated by Algorithm~\ref{alg:NFBEE}, 
\eqref{eq:stepsizesFPDHFE} corresponds to \eqref{eq:stepsizes},
and the conclusion follows by Theorem~\ref{teo:NFBEE}. 
% Finally, $\bm M_n+\bm{C}-\bm{S}$ is $\gamma\zeta(1-\tau\gamma\|L\|^2)^{-1/2}$-Lipschitz and $\bm C$ is $\beta(1-\tau\gamma\|L\|^2)$-cocoercive with respect to $\bm S$. Hence, the conclusion follows from Theorem~\ref{teo:NFBEE}, since $
% \zer(\bm A+\bm C)$
% coincides with the solution set of Problem~\ref{pro:mainPD}.
\end{proof}
\begin{rem}\label{rem:algFPDHFE}
\begin{enumerate}
    \item When $(e_n^1,e_n^2)=(0,0)$, Algorithm~\ref{alg:FPDHFE} reduces to the standard FPDHF.
     \item In the case where $C=0$, Algorithm~\ref{alg:FPDHFE} is an inexact version of the Condat--V\~u algorithm. If additionally $D=0$, it reduces to an inexact version of the Chambolle--Pock algorithm. 
%      In the latter case, if we additionally consider exact calculations of $J_{\tau B^{-1}}$, the algorithm reduces to
%     \begin{equation}
% 	(\forall n\in\N)\quad 
% 	\begin{array}{l}
% 		\left\lfloor
% 		\begin{array}{l}
%             \textnormal{ find } (p_n,y_n) \in \gra A  \textnormal{ such that }\\
% 			\left\lfloor
% 		\begin{array}{l} p_n^{*} = x_n/\gamma-L^*u_n-p_n /\gamma \\
%         e_n^1 = y_n-p_n^{*}\\
%            %\|(e_n^1,0)\|_{\bm{S}^{-1}}\leq \sigma \|(p_n-x_n,q_n-u_n)\|_{\bm{S}},
%            \|e_n^1\|\leq \sigma(1-\gamma\tau\|L\|^2) \|p_n-x_n\|
% 		\end{array}
% 		\right.\\
%         q_n = J_{\tau B^{-1}}(u_n+\tau L(2p_n+x_n))\\
% 			(x_{n+1},u_{n+1})=(p_n,q_n)+\gamma \bm{S}^{-1}(e_n^1,0).
% 		\end{array}
% 		\right.
% 	\end{array}
% \end{equation}
    \item % Note that the metrics induced by $\bm{S}$ and $\bm{S}^{-1}$ are used in Algorithms ~\ref{alg:FPDHF} and~\ref{alg:FPDHFE} for finding an adequate approximation of the resolvents. The calculation of these norms can be avoided in  numerical implementations, for example, if
    Note that Algorithms~\ref{alg:FPDHF} and~\ref{alg:FPDHFE} involve the norms induced by $\bm S$ and $\bm S^{-1}$ in the relative-error criterion used to compute approximate resolvent evaluations. In practice, the explicit computation of these norms can be avoided, for instance, if

        \begin{equation}\label{eq:errorPD2}
        \|e_n^1\|^2+\frac{\tau}{\gamma}\|e_n^2\|^2 \leq \sigma^2 (1-\sqrt{\gamma\tau}\|L\|)^2 \left(\|p_n-x_n\|^2+\frac{\gamma}{\tau}\|q_n-u_n\|^2\right).
    \end{equation}
    % by the rule defined in \eqref{eq:errorPD2} below. 
    Indeed, in view of Proposition~\ref{prop:operators}, we have for every  $(z,v) \in \HH$, 
    % \begin{equation}
    %     \bm{S}^{-1}  \colon \HH \to \HH \colon (z,v) \mapsto \big((\id-\gamma\T L^*L)^{-1}(z+\tau L^*v),\tau (\id-\gamma\T LL^*)^{-1}(Lz+v/\gamma)\big).
    % \end{equation}
    %Hence, for every  $(z,v) \in \HH$, 
    \begin{align}
    \|(z,v)\|_{\bm{S}^{-1}}^2%=\|((\id-\gamma\T L^*L)^{-1}(z+\tau L^*v)\|^2+\|\tau (\id-\gamma\T LL^*)^{-1}(Lz+v/\gamma)\|^2\\
     & = \scal{(1-\gamma \tau L^*L)^{-1}(z+\tau L^*v)}{z}+\tau\scal{(1-\gamma \tau LL^*)^{-1}(Lz+v/\gamma)}{v}\nonumber\\
     %& = \scal{(1-\gamma \tau L^*L)^{-1}z}{z}+2\tau\scal{(1-\gamma \tau LL^*)^{-1}Lz}{v}+\frac{\tau}{\gamma}\scal{(1-\gamma \tau LL^*)^{-1}v}{v}\\
     &\leq \|(1-\gamma \tau L^*L)^{-1}\|\|z+\tau L^* v\|\|z\|+\tau\|(1-\gamma \tau LL^*)^{-1}\|\|Lz+v/\gamma\|\|v\|\nonumber\\
    &\leq \frac{1}{1-\gamma \tau \|L\|^2}\left(\|z+\tau L^* v\|\|z\|+\tau\|Lz+v/\gamma\|\|v\|\right)\nonumber\\
       &\leq \frac{1}{1-\gamma \tau \|L\|^2}\left(\|z\|^2+2\tau\| L\|\| v\|\|z\| + \frac{\tau}{\gamma}\|v\|^2\right)\nonumber\\
     &\leq \frac{1}{1-\gamma \tau \|L\|^2}\left(\|z\|^2+\sqrt{\gamma\tau}\| L\|\|z\|^2 +\frac{\tau\sqrt{\tau}}{\sqrt{\gamma}}\| L\|\| v\|^2 + \frac{\tau}{\gamma}\|v\|^2\right)\nonumber\\
&\leq \frac{1+\sqrt{\gamma\tau}\|L\|}{1-\gamma \tau \|L\|^2}\left(\|z\|^2+ \frac{\tau}{\gamma}\|v\|^2\right)\nonumber\\
&= \frac{1}{1-\sqrt{\gamma \tau} \|L\|}\left(\|z\|^2+ \frac{\tau}{\gamma}\|v\|^2\right).\label{eq:boundSm1}
    \end{align}
  In addition, for every $(x,u) \in \HH$,
    \begin{align}
    \|(x,u)\|^2_{\bm{S}} &= \scal{(x-\gamma L^*u,\gamma u/\tau -\gamma Lx)}{(x,u)}\nonumber\\
    & = \|x\|^2-2\gamma\scal{Lx}{u}+\frac{\gamma}{\tau}\|u\|^2\nonumber\\
    & \geq \|x\|^2-\sqrt{\gamma \tau} \|L\|\|x\|^2 - \frac{\gamma\sqrt{\gamma}\|L\|}{\sqrt{\tau}}\|u\|^2 +\frac{\gamma}{\tau}\|u\|^2\nonumber\\
    &=(1-\sqrt{\gamma \tau} \|L\|)\left(\|x\|^2+\frac{\gamma}{\tau}\|u\|^2\right)\label{eq:boundS}.
    \end{align}
    Therefore, by \eqref{eq:boundSm1} and \eqref{eq:boundS}, the condition on the error in Algorithm~\ref{alg:FPDHFE} holds if \eqref{eq:errorPD2} holds.
    %\item  Algorithm~\ref{alg:FPDHF} and Algorithm~\ref{alg:FPDHFE} can be implemented using exact evaluations of $J_{\gamma A}$ or $J_{\tau B^{-1}}$. For instance, if $J_{\tau B^{-1}}$ is implemented exactly, then $e_n^2=0$, and the condition $\|(e_n^1,e_n^2)\|_{\bm{S}^{-1}}\leq \sigma \|(p_n-x_n,q_n-u_n)\|_{\bm{S}}$ holds if $\|e_n^1\|\leq \sigma(1-\gamma\tau\|L\|^2) \|p_n-x_n\|$. Indeed, it follows from \cite[Lemma~6.1]{MorinBanertGiselsson2022} that
    %\begin{align*}
     %   \|(e_n^1,0)\|_{\bm{S}^{-1}}^2 \leq \frac{1}{1-\gamma\tau \|L\|^2}\|e_n^1\|^2 \leq \sigma^2(1-\gamma\tau\|L\|^2) \|p_n-x_n\|^2\leq \sigma^2\|(p_n-x_n,q_n-u_n)\|^2_{\bm{S}}.
    %\end{align*}
    %Similarly, if $J_{\gamma A}$ is implemented, we have $e_n^1=0$ and if $\|e_n^2\|\leq \frac{\gamma(1-\gamma\tau\|L\|^2)}{\tau}\|q_n-u_n\|^2$, the condition on the error holds.
    \item\label{rem:algFPDHFE5}  Similarly to Remark~\ref{rem:FPDHFunaresolvente}.\ref{rmk2}, it is possible to completely avoid the operator $\bm{S}^{-1}$ in Algorithm~\ref{alg:FPDHFE} if we consider $v_n = (u_n-q_n)/\tau+L(p_n-\gamma(y_n+Cx_n+Dp_n+L^*u_n))$. In that case, 
%     we have
%     \begin{align*}
%     e_n^2 &= v_n-q_n^*\\
%          &= \frac{u_n-q_n}{\tau}+L(p_n-\gamma(y_n+Cx_n+Dp_n+L^*u_n)) - L(p_n+s_n-x_n)+\frac{q_n-u_n}{\tau}\\
%           &= -\gamma L \left( y_n -\left(\frac{x_n-p_n}{\gamma} -Cx_n-Dx_n-L^*u_n\right)\right)   \\
%           &= -\gamma L \left( y_n -p_n^*\right).
%     \end{align*}
% Therefore, $(e_n^1,e_n^2)=\bm{S}(y_n-p_n^*,0)$ and 
% \begin{align*}
%     \|(e_n^1,e_n^2)\|_{\bm{S}^{-1}} = \|(y_n-p_n^*,0)\|_{\bm{S}}=\|y_n-p_n^*\|. %=  \left \|y_n -\frac{x_n-p_n}{\gamma} -Cx_n-Dx_n-L^*u_n\right\|.
% \end{align*}
% Moreover, 
% \begin{align*}
%     v_n \in B^{-1} q_n &\Leftrightarrow \frac{u_n-q_n}{\tau}+L(p_n-\gamma(y_n+Cx_n+Dp_n+L^*u_n)) \in B^{-1} q_n\\
%     &\Leftrightarrow q_n = J_{\tau B^{-1}} \left( u_n +\tau L(p_n-\gamma(y_n+Cx_n+Dp_n+L^*u_n))\right)\\
%     &\Leftrightarrow  q_n = J_{\tau B^{-1}} \left( u_n +\tau L(s_n+p_n-x_n-\gamma (y_n-p_n^*))\right)
% \end{align*}
% Hence, 
Algorithm~\ref{alg:FPDHFE} can be written as follows
    \begin{equation}\label{eq:algFPDHFEexplicit}
	(\forall n\in\N)\quad 
	\begin{array}{l}
		\left\lfloor
		\begin{array}{l}
            \textnormal{ find } (p_n,y_n) \in \gra A \\
			\left\lfloor
		\begin{array}{l} p_n^{*} = x_n/\gamma - Cx_n-Dx_n-L^*u_n-p_n /\gamma \\
        s_n = p_n+\gamma (Dx_n-Dp_n) \\
q_n = J_{\tau B^{-1}} \left( u_n +\tau L(s_n+p_n-x_n-\gamma (y_n-p_n^*))\right)\\
           %\left \|\gamma y_n -x_n-p_n -\gamma (Cx_n+Dx_n+L^*u_n)\right\| \leq \sigma\gamma  \|(p_n-x_n,q_n-u_n)\|_{\bm{S}},	
\| y_n -p_n^* \|\leq \sigma  \|(p_n-x_n,q_n-u_n)\|_{\bm{S}},	
		\end{array}
		\right.\\
			(x_{n+1},u_{n+1})=(s_n+\gamma(y_n-p_n^*),\;q_n).
		\end{array}
		\right.
	\end{array}
\end{equation}
In view of \cite[Lemma~6.1]{MorinBanertGiselsson2022}, $\| y_n -p_n^* \|\leq \sigma  \|(p_n-x_n,q_n-u_n)\|_{\bm{S}}$ holds when $\|y_n-p_n^*\|\leq \sigma \sqrt{(1-\gamma \tau \|L\|^2)}\|p_n-x_n\|$. Hence, no explicit computation of the $\bm S$-norm is required in practical implementations. Finally, when $C=D=0$, Algorithm~\eqref{eq:algFPDHFEexplicit} reduces to the inexact Chambolle--Pock method proposed in \cite[Algorithm~3]{AlvesLorenzNaldi2026}.
\end{enumerate}
\end{rem}
\section{Applications and numerical experiments}\label{sec:applications}
In this section, we provide practical applications and numerical experiments to show the advantages of incorporating approximations in the computation of the resolvent. We present two classes of problems: saddle-point and convex optimization problems. While both formulations can be mathematically linked via Fenchel-Rockafellar duality, they represent fundamentally different modeling paradigms that necessitate separate treatment. The saddle-point formulation is native to adversarial and equilibrium-seeking settings, such as zero-sum games, robust optimization, and generative adversarial networks. In these contexts, the dual variable is not an auxiliary construct, but a primary decision entity (e.g., a competing player or an adversary) endowed with its own structural constraints. 

 \subsection{Saddle-Point Problems}
In this subsection, we consider the numerical solution of the following saddle-point problem.
\begin{pro}\label{prob:saddle}
Let $f \in \Gamma_0(\H)$, $g \in \Gamma_0(\G)$, $L \colon \H\to \G$ be a bounded linear operator. The problem is to
		\begin{equation}\label{eq:problemsaddle}
			\min_{x\in \H}\max_{y \in \G} f(x)+ \scal{Lx}{y}-g(y),
		\end{equation}
        under the assumption that its solution set is nonempty.
\end{pro}
This problem encompasses several applications such as zero-sum games \cite{vonNeumannMorgenstern1944}, robust optimization \cite{RobustOpti2009}, generalized lasso problems \cite{Tibshirani1996,Tibshirani2011}, and generative adversarial networks \cite{gidel2019variational,Mescheder2017}, among others. Problem~\ref{prob:saddle} is equivalent to 
		\begin{equation}
			\text{find} \quad (x,y) \in \H\times \G \quad \text{ such that } \quad (0,0) \in A(x,y) + D(x,y),
		\end{equation}
where 
\begin{align*}
    A \colon  \H \times \G \to 2^{\H \times \G} \colon (x,y)\mapsto \partial f(x) \times \partial g(y)\\
    D \colon  \H \times \G \to {\H \times \G} \colon (x,y)\mapsto (L^*y,-Lx).
\end{align*}
We have that $A$ is maximally monotone and $D$ is monotone and $\|L\|$-Lipschitz, thus, Problem~\ref{prob:saddle} can be solved by the FBF algorithm and its inexact resolvent versions, namely, Algorithm~\ref{alg:FBHFEP} (IFBF {\it inexact FBF}) and Algorithm~\ref{alg:FBHFE} (EIFBF {\it explicit inexact FBF}) for $C=0$. Particularly, Algorithm~\ref{alg:FBHFEP} reduces to  \begin{equation}\label{eq:algFBHFEP_saddle}
	(\forall n\in\N)\quad 
	\begin{array}{l}
		\left\lfloor
		\begin{array}{l}
            \textnormal{ find } (z_n^1,y_n^1) \in \gra \partial f \textnormal{ and }  (z_n^2,y_n^2) \in \gra \partial g
            \textnormal{ such that }\\
			\left\lfloor
		\begin{array}{l} (z_n^{1,*},z_n^{2,*}) = (x_n^1/\gamma -L^*x^2_n-z_n^1 /\gamma ,x_n^2/\gamma +Lx^1_n-z_n^2 /\gamma)\\
       (e_n^1, e_n^2) = (y_n^1  - z_n^{1,*} ,y_n^2  - z_n^{2,*}) \\
           \|e_n^1\|^2+\|e_n^2\|^2 \leq \sigma^2 (\|z_n^1-x_n^1\|^2+\|z_n^2-x_n^2\|^2),	
		\end{array}
		\right.\\
        (t_n^{1,*},t_n^{2,*}) = (y_n^1+L^*z_n^2,y_n^2-Lz_n^1)\\
			\delta_n = \scal{x_n^1-z_n^1}{t_n^{1,*}}+ \scal{x_n^2-z_n^2}{t_n^{2,*}}\\
			(d_n^1,d_n^2) = \begin{cases}
				\dfrac{\delta_n}{\|(t_n^{1,*},t_n^{2,*})\|^2} (t_n^{1,*},t_n^{2,*}), &\textnormal{ if }
				\delta_n >0; \\
				0,  &\textnormal{ otherwise } 
			\end{cases}\\
			(x_{n+1}^1,x_{n+1}^2)=(x_n^1,x_n^2)-\lambda_n(d_n^1,d_n^2).
		\end{array}
		\right.
	\end{array}
\end{equation}
Similarly, it is possible to derive the explicit version of Algorithm~\ref{alg:FBHFE}, which is omitted here for the sake of conciseness. To exhibit the advantages of allowing approximations in the computation of the resolvent, we consider the following numerical example. 
\subsubsection{Numerical implementation} To numerically\footnote{These experiments were implemented in MATLAB on a desktop computer equipped with an Intel Core i7-14700K processor (3.4/5.6~GHz), 64~GB of RAM, and running Windows~11 Pro 64-bit. The code is available in this  \href{https://drive.google.com/file/d/1oANDJOOHLyJTHgXGNHk3eIyoDUbd2RHf/view?usp=drive_link}{repository}} compare Algorithm~\ref{alg:FBHFEP} and Algorithm~\ref{alg:FBHFE} with the standard FBF, we consider Problem~\ref{prob:saddle} when $\H =\R^N$, $\G =\R^M$, $f (x) = \frac{1}{2}x^\top Q x + q^\top x$, $g(y) = \iota_{[-1,1]^M}(y)$, $L \in \R^{M\times N}$, $Q \in \mathbb{R}^{N \times N}$ is a symmetric positive definite matrix, and $q \in \R^N$.
Note that, at each iteration, FBF needs to calculate the resolvent of $A$ which is given by
\begin{equation*}
    (\forall (x,y) \in \H \times \G) \quad J_{\gamma A}(x,y) = (\prox_{\gamma f} x, \prox_{\gamma g} y ) = ( (\id+\gamma Q)^{-1} (x-\gamma q), P_{[-1,1]^M}y ). 
\end{equation*}
Since $ P_{[-1,1]^M}y = \max(-1, \min(1, y))$, $\prox_{\gamma g}$ can be easily implemented. On the other hand, $\prox_{\gamma f}$ is numerically expensive in high dimensions, to avoid this costly implementation, we  consider approximations on the resolvent  by the conjugate gradient method \cite{hestenes1952methods} for solving the system $z_n^1 = (\id+\gamma Q)^{-1} (x_n^1-\gamma (L^*x_n^2+q))$. Then, in \eqref{eq:algFBHFEP_saddle}, we set, for every $n \in \N$, $e_n ^2=  0$; thus, $z_n^2 = \prox_{\gamma g} (x_n^2+\gamma L x_n^1)$. Defining  $\mathcal{M} = \id + \gamma Q$, we approximate $(z_n^1,y_n^1)$ using the following subroutine: initialize $p^0 = x_n^1$, $r^0= x_n^1 -\gamma (L^* x_n^2+q) - \mathcal{M} p^0$, $d^0 = r^0$,  and
\begin{equation}\label{eq:alg_saddle_Subroutine}
	\begin{array}{l}
		(\forall k\in\N)\quad
		\left\lfloor
		\begin{array}{l} 
			\textnormal{if } \quad \|r^{k}\| \leq \sigma\|p^k-x_n^1\|  \quad\textnormal{ return} \\[2mm]
			\quad \left\lfloor
			\begin{array}{l}
				z_n^1 = p^{k} \quad \textnormal{and} \quad\gamma(y_n^1 - z_n^{1,*}) = -r^{k} 
			\end{array}
			\right.\\[4mm]
            \textnormal{otherwise}\\[2mm]
			\left\lfloor
		\begin{array}{l} \alpha_k = \dfrac{\|r^{k}\|^2}{(d^{k})^\top {\mathcal{M} d^{k}}} \\
			p^{k+1} = p^{k} + \alpha_k d^{k} \\
			r^{k+1} = r^{k} - \alpha_k \mathcal{M} d^{k}\\
			\eta_k = \dfrac{\|r^{k+1}\|^2}{\|r^{k}\|^2} \\
			d^{k+1} = r^{k+1} + \eta_k d^{k}.
            \end{array}
			\right.
		\end{array}
		\right.
	\end{array}
\end{equation}
To test the algorithms, we consider nine pairs of dimensions $(N,M)$ described in Table~\ref{T:results_saddle}. For each value of $(N,M)$, we generate 20 random instances of $Q$, $q$, and $L$ by using the {\it randn} function from MATLAB. We ran FBF with step-size $\gamma = 0.99/\|L\|$ and IFBF and EIFBF with $\gamma = 0.99/(\|L\|+\sigma)$ with $\sigma \in \{0.1,0.5,0.9\}$. The algorithms stop when a limit of $10^5$ iterations is reached or when the relative error is less than $10^{-6}$. For each pair $(N,M)$, the results in terms of the average number of iterations and average CPU time over the 20 realizations, are presented in Table~\ref{T:results_saddle}. From this table we can observe the numerical advantages of incorporating approximation in the resolvent. In every case, FBF is outperformed by the inexact versions and the best performance is by IFBF with $\sigma = 0.9$, which reduces the CPU time by more than 55\% compared to FBF. Note that, as $\sigma$ is larger, the number of subiterations decreases, as expected. 
\begin{table}[h!]
    \centering
   \begin{tabular}{l c ccc ccc ccc}
        \toprule
        % ---------------------------------------------------
        % BLOQUE N = 500
        % ---------------------------------------------------
        %\multicolumn{11}{c}{$\bm{N = 500}$} \\
        %\midrule
        \multicolumn{2}{c}{$\bm{N = 500}$}  & \multicolumn{3}{c}{$\bm{M = 150}$} 
        & \multicolumn{3}{c}{$\bm{M = 250}$} 
        & \multicolumn{3}{c}{$\bm{M = 400}$} \\
        \cmidrule(lr){3-5} \cmidrule(lr){6-8} \cmidrule(lr){9-11}
        \textbf{Algorithm} & $\bm{\sigma}$ & \textbf{NI} & \textbf{T} & \textbf{SI} 
                                           & \textbf{NI} & \textbf{T} & \textbf{SI} 
                                           & \textbf{NI} & \textbf{T} & \textbf{SI} \\
        \midrule
        FBF   & --  & 580 & 0.88 & --    & 1780 & 2.85 & --    & 13382 & 34.79 & --    \\
        \addlinespace
        \multirow{3}{*}{IFBF} 
              & 0.1 & 601 & 0.21 & 24  & 1875 & 0.76 & 23  & 14970& 18.14 & 22 \\
              & 0.5 & 599  & 0.20 & 17  & 1870  & 0.72 & 16  & 14930 & 17.61 & 15 \\
              & 0.9 & 597  & {\bf 0.18} & 15  & 1865 & {\bf 0.68} & 14  & 14885 & {\bf 17.41} & 13  \\
        \addlinespace
        \multirow{3}{*}{EIFBF} 
              & 0.1 & 581  & 0.23  & 23 & 1784  & 0.82 & 20  & 13407  & 19.30 & 15\\
              & 0.5 & 587  & 0.22  & 18  & 1800  & 0.79  & 14  & 13507 & 19.11 & 10  \\
              & 0.9 & 593 & 0.20  & 16  & 1816  & 0.76  & 12  & 13607  & 19.10& 8 \\
        
        \midrule
        % ---------------------------------------------------
        % BLOQUE N = 1000
        % ---------------------------------------------------
        %\multicolumn{11}{c}{$\bm{N = 1000}$} \\
        %\midrule
        \multicolumn{2}{c}{$\bm{N = 1000}$} & \multicolumn{3}{c}{$\bm{M = 300}$} 
        & \multicolumn{3}{c}{$\bm{M = 500}$} 
        & \multicolumn{3}{c}{$\bm{M = 800}$} \\
        \cmidrule(lr){3-5} \cmidrule(lr){6-8} \cmidrule(lr){9-11}
        \textbf{Algorithm} & $\bm{\sigma}$ & \textbf{NI} & \textbf{T} & \textbf{SI} 
                                           & \textbf{NI} & \textbf{T} & \textbf{SI} 
                                           & \textbf{NI} & \textbf{T} & \textbf{SI} \\
        \midrule
        FBF   & --  & 812 & 5.66 & --    & 2466 & 19.83 & --    & 11296 & 101.33 & --    \\
        \addlinespace
        \multirow{3}{*}{IFBF} 
              & 0.1 & 858 & 1.73 & 30 & 2674 & 6.60 & 28  & 12918 & 38.15 & 27  \\
              & 0.5 & 856  & 1.59 & 22  & 2670 & 6.20 & 20 & 12896 & 36.45& 20  \\
              & 0.9 & 855 & {\bf 1.54} & 19  & 2665 & {\bf 6.09} & 18 & 12869 & {\bf 35.74} & 17 \\
        \addlinespace
        \multirow{3}{*}{EIFBF} 
              & 0.1 & 813  & 2.09 & 28 & 2469  & 7.94 & 24 & 11311 & 43.16 & 20  \\
              & 0.5 & 819  & 1.96 & 22  & 2485  & 7.63 & 18  & 11372  & 42.04 & 13\\
              & 0.9 & 825 & 1.94 & 19  & 2501  & 7.56 & 15  & 11433 & 41.79 & 11 \\
        
        \midrule
        % ---------------------------------------------------
        % BLOQUE N = 2000
        % ---------------------------------------------------
        %\multicolumn{11}{c}{$\bm{N = 2000}$} \\
        %\midrule
        \multicolumn{2}{c}{$\bm{N = 2000}$} & \multicolumn{3}{c}{$\bm{M = 600}$} 
        & \multicolumn{3}{c}{$\bm{M = 1000}$} 
        & \multicolumn{3}{c}{$\bm{M = 1600}$} \\
        \cmidrule(lr){3-5} \cmidrule(lr){6-8} \cmidrule(lr){9-11}
        \textbf{Algorithm} & $\bm{\sigma}$ & \textbf{NI} & \textbf{T} & \textbf{SI} 
                                           & \textbf{NI} & \textbf{T} & \textbf{SI} 
                                           & \textbf{NI} & \textbf{T} & \textbf{SI} \\
        \midrule
        FBF   & --  & 3249 & 98 & --   & 5619 & 179.74 & --   & 22513 & 784.62 & --   \\
        \addlinespace
        \multirow{3}{*}{IFBF} 
              & 0.1 & 3618 & 39.77  & 35 & 6390 & 76.62 & 35 & 27343 & 394.90 & 34 \\
              & 0.5 & 3614 & 40.16  & 26 & 6383 & 73.69  & 25 & 27307 & 358.92& 25 \\
              & 0.9 & 3610 & {\bf 38.55}  & 23 & 6375 & {\bf 71.41}  & 22 & 27270 & {\bf 351.58} & 22 \\
        \addlinespace
        \multirow{3}{*}{EIFBF} 
              & 0.1 & 3253& 42.51  & 31 & 5624 & 81.48  & 29 & 22534 & 389.57  & 24.57 \\
              & 0.5 & 3615 & 42.53  & 22 & 5647 & 79.30  & 21 & 22613 & 369.00  & 17\\
              & 0.9 & 3610  & 41.52  & 19& 5670 & 77.69  & 18 &  22693 & 365.04  & 14 \\
        \bottomrule
    \end{tabular}
    \caption{Results in terms of, number of iterations (NI), CPU time in seconds (T), and mean value of subiterations (SI). For each dimension, the best CPU time is highlighted in {\bf black}.}
    \label{T:results_saddle}
\end{table}
 \subsection{Convex Optimization Problems}
 We focus on the following convex optimization problem.
	\begin{pro}\label{prob:problemopti}
		%In the context of Problem~\ref{pro:main}, 
		Let $L\colon \H \to \G$ be a
		bounded linear operator, let $f \in \Gamma_0(\H)$, $h \in \Gamma_0(\H)$, and $g\in \Gamma_0(\G)$. Suppose that $h$ is differentiable with a $(1/\beta)$-Lipschitz continuous
		gradient. The problem is to 
		\begin{equation}\label{eq:problemopti}
			\min_{x\in \H}f(x)+g(Lx)+h(x)
		\end{equation}
		under the assumption that its solution set is nonempty.
	\end{pro}
	This optimization problem encompasses applications in data science \cite{CombettesPesquet2021strategies}, machine learning \cite{Nocedal2018}, image processing
	\cite{BotHendrich2014TV,Briceno2011ImRe,BurgerSawatzkySteidl2014,chambolle2016AN}, among others.
	If $0 \in \sri(\dom g - L(\dom f))$, by considering $A=\partial f$, $B= \partial g$, $C= \nabla h$, and $D=0$, Problem~\ref{prob:problemopti} is a particular instance of Problem~\ref{pro:mainPD} \cite[Theorem~27.2]{bauschkebook2017}. Therefore, this problem can be solved by the Condat--V\~u algorithm, which needs to evaluate $\prox_{\gamma f}$ and $\prox_{\tau g^*}$, where $(\gamma,\tau) \in \RPP^2$. In the case that $\prox_{\gamma f}$ has a high computational cost, it is desirable to consider an approximation of it to decrease the computational time. 
	In this scenario, %considering the operators $A$, $B$, $C$, and $D$ defined above,
	Algorithm~\ref{alg:FPDHFE} with exact resolvent on $B$ (see \eqref{eq:algFPDHFEexplicit}) reduces to
	\begin{equation}\label{eq:algCVEexplicit}
		(\forall n\in\N)\quad 
		\begin{array}{l}
			\left\lfloor
			\begin{array}{l}
				\textnormal{ find } (p_n,y_n) \in \gra \partial f 
				\\
				\left\lfloor
				\begin{array}{l} p_n^{*} = x_n/\gamma - \nabla h (x_n)-L^*u_n-p_n /\gamma \\
					q_n = \prox_{\tau g^*} \left( u_n +\tau L(2p_n-x_n-\gamma (y_n-p_n^*))\right)\\
					%\left \|\gamma y_n -x_n-p_n -\gamma (Cx_n+Dx_n+L^*u_n)\right\| \leq \sigma\gamma  \|(p_n-x_n,q_n-u_n)\|_{\bm{S}},	
					\| y_n -p_n^* \|\leq \sigma  \|(p_n-x_n,q_n-u_n)\|_{\bm{S}},	
				\end{array}
				\right.\\
				(x_{n+1},u_{n+1})=(p_n-\gamma(y_n-p_n^*),\;q_n).
			\end{array}
			\right.
		\end{array}
	\end{equation}
	Next, we compare this algorithm with the standard Condat--V\~u algorithm and with its summable errors version in the context of Computed Tomography (CT) reconstruction problems \cite{Kak2001}. In particular, let $\H = \R^N$, $\G= \R^M$, and consider $c \in \R^M$ a noisy tomographic projection of an image $\overline{x} \in \R^N$. The objective is to reconstruct $\overline{x}$ from the observation $c$. We assume that $
	c = T(\overline{x})+ \epsilon,
	$
	where $T \in \R^{M\times N}$ is the discretized Radon projector and $\epsilon$ represents Gaussian noise. An approach to recover $\overline{x}$ is to solve the following optimization problem: 
	\begin{equation}\label{eq:problemoCT}
		\min_{x\in \R^N} F(x):=\frac{1}{2}\|Tx-c\|_2^2 + \lambda_1 H_{\delta}(W x)+\lambda_2\|\nabla x\|_1
	\end{equation}
	%Problem~\ref{prob:problemopti} with $f = \lambda_1 H_{\delta}\circ W$, $g = \lambda_2\|\nabla \cdot\|_1$, and $h = \frac{1}{2}\|T(\cdot)-c\|_2^2$, 
	where $W\in \R^{N \times N}$ is an orthonormal wavelet transform, $\nabla$ is the discrete gradient with Neumann boundary conditions, $(\lambda_1,\lambda_2) \in \RPP^2$ are regularization parameters, and
	for $\delta >0$ and $x = (x_i)_{1\leq i \leq N} \in \R^N$, $H_{\delta}$ is the Huber function given by 
	\begin{equation}\label{eq:def_huber}
		H_{\delta}(x) =  \sum_{i=1}^N \phi_{\delta}(x_i) \  \textnormal{ and } \ 
		(\forall \eta \in \R) \quad
		\phi_{\delta} (\eta) = \begin{cases}
			|\eta|-\frac{\delta}{2}, & \text{ if } |\eta| > \delta,\\
			\frac{\eta^2}{2\delta}, & \text{ otherwise}.
		\end{cases}
	\end{equation}
	The functions $\lambda_1 H_{\delta}\circ W$ and $\lambda_2\|\nabla \cdot\|_1$ promote sparsity of the image to be recovered while  $\frac{1}{2}\|T(\cdot)-c\|_2^2$ acts as the data fidelity term. We have that $H_\delta$ is differentiable and its gradient is $(1/\delta)$-Lipschitz continuous. Explicit formulas for $\nabla H_\delta$ and $\prox_{H_\delta}$ can be found in \cite{BricenoPustelnik2023} and, by the orthogonality of $W$, $\prox_{H _\delta \circ W }$ is also explicit in view of \cite[Corollary~23.27]{bauschkebook2017}. For additional details on the model, the reader is referred to \cite{TV-chambolle,chambolle1997,Klann_2015,Loris2007,PustelnikWavelet,ROF1992}. 
	
	Since $\prox_{H_\delta \circ W}$ admits a closed-form expression, the problem can be solved by Condat--V\~u, that is, by the algorithm in \eqref{eq:algCVEexplicit} with $f = \lambda_1 H_{\delta}\circ W$, {$g = \lambda_2\|\cdot\|_1$, $L=\nabla$}, $h = \frac{1}{2}\|T(\cdot)-c\|_2^2$, and $\sigma =0$. Note that, in this case, $C$ is $\|T\|^{-2}$-cocoercive. Since $\|T\|$ is generally large for CT problems, the stepsizes are forced to be small, which usually slows down the convergence. To allow larger stepsizes in Condat--V\~u, it can also be applied with $f = \frac{1}{2}\|T(\cdot)-c\|_2^2$, {$g = \lambda_2\|\cdot\|_1$, $L=\nabla$}, and $h = \lambda_1 H_{\delta}\circ W$, but calculating $\prox_{\gamma f} (x)= (\id +\gamma T^*T)^{-1}(x{+}\gamma T^*c)$ requires solving a large linear system. To avoid solving this linear system, we consider an approximate solution by the conjugate gradient method \cite{hestenes1952methods}, which is stopped as soon as the error criterion in Remark~\ref{rem:algFPDHFE}.\ref{rem:algFPDHFE5} holds.
	%Setting \textcolor{blue}{$\partial f=\nabla f = T^*T(\cdot)-T^*c$} in \eqref{eq:algCVEexplicit}, we then have $y_n = T^*T p_n-T^*c$ and 
	%\begin{equation*}
	%    \gamma(y_n -p_n^*) =  (\id + \gamma T^*T) p_n -x_n+\gamma (\nabla h (x_n)+L^*u_n)-\gamma T^*c.
	%\end{equation*}
	%Moreover, define $\mathcal{M} = \id + \gamma T^*T $, $b_n = x_n - \gamma(\nabla h(x_n) + L^*u_n)\textcolor{blue}{+}\gamma T^*c$, $p^0 = x_n$, $r^0= b_n - \mathcal{M} p^0$, and $d^0 = r^0$. Then, $(p_n,y_n)$ are chosen according to the following subroutine.
	%\begin{equation}\label{eq:algCG_Subroutine}
	%	\begin{array}{l}
		%		(\forall k\in\N)\quad
		%		\left\lfloor
		%		\begin{array}{l} 
			%			\textnormal{if } \quad \|r^{k}\| \leq \textcolor{blue}{\gamma}\sigma \sqrt{(1-\gamma \tau \|L\|^2)}\|p^k-x_n\|  \quad\textnormal{ return} \\[2mm]
			%			\quad \left\lfloor
			%			\begin{array}{l}
				%				p_n = p^{k} \quad \textnormal{and} \quad\gamma(y_n - p_n^*) = -r^{k} 
				%			\end{array}
			%			\right.\\[4mm]
			%            \textnormal{otherwise}\\[2mm]
			%			\left\lfloor
			%		\begin{array}{l} \alpha_k = \dfrac{\|r^{k}\|^2}{(d^{k})^\top {\mathcal{M} d^{k}}} \\
				%			p^{k+1} = p^{k} + \alpha_k d^{k} \\
				%			r^{k+1} = r^{k} - \alpha_k \mathcal{M} d^{k}=\textcolor{blue}{b_n-\mathcal{M}p^{k+1}} \\
				%			\eta_k = \dfrac{\|r^{k+1}\|^2}{\|r^{k}\|^2} \\
				%			d^{k+1} = r^{k+1} + \eta_k d^{k}
				%            \end{array}
			%			\right.
			%		\end{array}
		%		\right.
		%	\end{array}
	%\end{equation}
	%where the error criterion is satisfied in view of Remark~\ref{rem:algFPDHFE}.\ref{rem:algFPDHFE5}.
	Now, we numerically compare three algorithms, whose settings are summarized in Table~\ref{T:algorithms}. CV1 is the standard Condat--V\~u algorithm. The other two algorithms consider $f=\frac{1}{2}\|T(\cdot)-c\|_2^2$ and compute $p_n$ by the conjugate gradient method applied to the linear system $(\id+\gamma T^*T)p=b_n$, where $b_n=x_n-\gamma(\nabla h(x_n)+L^*u_n)+\gamma T^*c$, initialized at $x_n$. Denote by $p^k$ and $r^k$ the approximate solution and the residual of this system at the $k$-th subiteration, respectively. EICV ({\it explicit inexact Condat--V\~u}) is the algorithm in \eqref{eq:algCVEexplicit}, in which the conjugate gradient method is stopped when $\|r^k\|\leq\gamma\sigma\sqrt{1-\gamma\tau\|L\|^2}\,\|p^k-x_n\|$, which guarantees the error criterion in Remark~\ref{rem:algFPDHFE}.\ref{rem:algFPDHFE5}. CV-sum is the Condat--V\~u algorithm with inexact proximal steps and summable errors \cite{Condat13,Vu13}, in which the conjugate gradient method is stopped when $\|r^k\|\leq\varepsilon_n:=\varepsilon_0/n^{1.1}$. %Since the eigenvalues of $\id+\gamma T^*T$ are not smaller than $1$, the distance from $p^k$ to the exact solution of the linear system is bounded by $\|r^k\|$ and, hence, the errors of CV-sum are summable.
	We set $\varepsilon_0=10^{-4}\|r^0\|$, where $r^0$ is the initial residual in the first iteration.
	\begin{table}[h!]
		\centering
		\setlength{\tabcolsep}{2.5pt}
		\begin{tabular}{c|c|c|c|c|c|c|c|c}
			{\bf Alg. } & Eq. & $f$ & $g$ & $h$ &  $\gamma$ & $\tau$ & $\beta $ \\ \hline
			CV1 & \eqref{eq:algCVEexplicit} & $\lambda_1 H_{\delta}\circ W$ & $\lambda_2\|\cdot\|_1$  & $  \frac{1}{2}\|T(\cdot)-c\|^2$ & $ 0.99\cdot2\kappa \beta $ & $ \frac{1}{\gamma \|\nabla\|^2}(1-\frac{\gamma}{2\beta}) $ & $ \frac{1}{\|T\|^2} $ \\
			CV-sum & \cite{Condat13,Vu13} & $  \frac{1}{2}\|T(\cdot)-c\|^2$  & $\lambda_2\|\cdot\|_1$  & $\lambda_1 H_{\delta}\circ W$  & $0.99\cdot2\kappa \beta $  & $  \frac{0.99}{\gamma \|\nabla\|^2}(1-\frac{\gamma}{2\beta}) $ & $ \frac{\delta}{\lambda_1}  $ \\
			EICV& \eqref{eq:algCVEexplicit} & $  \frac{1}{2}\|T(\cdot)-c\|^2$  & $\lambda_2\|\cdot\|_1$  & $\lambda_1 H_{\delta}\circ W$  & $0.99\cdot2\kappa \beta \cdot \varepsilon$  & $  0.99\,\frac{2\beta\varepsilon-\gamma}{2\beta\varepsilon\gamma \|\nabla\|^2} $ & $ \frac{\delta}{\lambda_1}  $
		\end{tabular}
		\caption{Settings of the algorithms to be compared. We define $\varepsilon = 2/(1+\sqrt{1+16\sigma^2\beta^2})$. We test these algorithms for $\kappa\in\{0.8,0.9,0.99\}$ in CV1, $\kappa=0.8$ in CV-sum and EICV, and $\sigma\in\{0.1,0.5,0.9\}$ in EICV. %The stepsizes $\gamma$ and $\tau$ were chosen as large as possible (up to the factor $0.99$) satisfying the conditions guaranteeing convergence in \eqref{eq:stepsizesFPDHFE} when $\zeta = 0$.
		}
		\label{T:algorithms}
	\end{table} 
	
	The discrete Radon transform was implemented using the ASTRA toolbox \cite{vanAarleASTRA2016,vanAarleASTRA2015}. As test images we considered the modified Shepp--Logan phantom of size $n\times n$, with $n\in\{128,256\}$ (see Figure~\ref{fig:phantom} for $n=128$). For the projector $T$ we considered a 2D fan-beam geometry with $90$ (for $n=128$) and $180$ (for $n=256$) projection angles uniformly distributed over the interval $[0, \pi[$. For $n=128$, the source-to-origin distance and the origin-to-detector distance were set to $800$ and $400$, respectively; for $n=256$, these distances and the number of detector cells were scaled proportionally to $n$, so that the field of view covers the whole image. The forward projection operator was explicitly constructed using a line-length projection model and stored on the GPU as a sparse matrix in double precision, so that $T^*$ is exactly its transpose. The observation $c$ was obtained by adding to $T\overline{x}$ a Gaussian noise with standard deviation $0.01\cdot\textnormal{mean}(T\overline{x})$ and, for each $n$, we considered five independent noise realizations. The wavelets transform is generated with a Symmlet basis of level $2$ and we considered the parameters $(\lambda_1,\lambda_2,\delta) = (10^{-4},10^{-2}, 10^{-5})$. All the algorithms start from $x_0=T^*c$ and $u_0=0$.
	
	%In order to measure the accuracy of all the algorithms in the same way, independently of their stepsizes, 
	We use the relative objective gap as stopping criterion, that is, each algorithm stops as soon as $
	(F(x_n)-F^\star)/|F^\star|\leq 10^{-4}$
	or when its computational time reaches $T_{\max}$, where $T_{\max}=40$~s for $n=128$ and $T_{\max}=300$~s for $n=256$. Here, $F^\star$ is the optimal value of \eqref{eq:problemoCT}, and the time needed to evaluate $F$ is not taken into account. Since $F^\star$ is unknown, for each image size and noise realization it was approximated by running EICV with $\kappa=0.8$ and $\sigma=0.1$ until the decrease of the objective function was smaller than $10^{-8}|F(p_n)|$ in 200 iterations.% in $200$ iterations was smaller than $10^{-8}|F(p_n)|$, keeping the smallest objective value found. We verified a posteriori that none of the iterates generated by the compared algorithms attains a smaller objective value.
	
	The results\footnote{These experiments were implemented in Python using PyTorch on a desktop computer equipped with an Intel Core i7-14700K processor (3.4/5.6~GHz), 64~GB of RAM, and an NVIDIA GeForce RTX~4060 GPU (8~GB), running Windows~11 Pro 64-bit. The code is available in this {\href{https://github.com/FernandoRoldanC/FernandoRoldanC.github.io/blob/main/files/ct_reconstruction_Inexact_warped_resolvent.ipynb}{repository}}.}  of our experiments are presented in Table~\ref{T:results}, where we report the average over the five noise realizations. From this table we can observe that CV1 and CV-sum do not reach the tolerance within $T_{\max}$ in any realization. In the case of CV1, this occurs because the cocoercivity constant is $\beta = \|T\|^{-2}$ ($\beta\approx 4.5 \cdot 10^{-5}$ for $n=128$ and $\beta\approx 1.1\cdot 10^{-5}$ for $n=256$), forcing the stepsizes to be small. Note that when $h=\lambda_1H_\delta \circ W$, the cocoercive constant is $\beta = \delta/\lambda_1=0.1$. CV-sum performs, on average, one subiteration per iteration, but its tolerances $\varepsilon_n$ are fixed in advance and do not depend on the progress of the algorithm, which results in a slow decrease of the objective gap. On the other hand, EICV reaches the tolerance in every instance and realization. In particular, CV1 and CV-sum did not reach the stopping criterion even with a time budget of about $3.7$ ($n=128$) and $4.5$ ($n=256$) times the average computational time of the best instance of EICV ($\kappa = 0.8$ and $\sigma = 0.9$). From Table~\ref{T:results} we also observe that the number of iterations of EICV is almost independent of $\sigma$, while larger values of $\sigma$ require fewer subiterations: from $\sigma=0.1$ to $\sigma=0.9$, the mean number of subiterations decreases by $33\%$ for both sizes, and the computational time decreases by $24\%$ ($n=128$) and $30\%$ ($n=256$). In Figure~\ref{fig:grap1}, we plot the relative objective gap versus time for the best instance of each algorithm in the first noise realization, corroborating our previous observations. The reconstructed images for $n=128$ are shown in Figure~\ref{fig:rec}. These results illustrate the numerical advantages of considering approximations of the resolvents under relative-error criteria.
	\begin{table}[h!]
		\centering
		\setlength{\tabcolsep}{4pt}
		\begin{tabular}{l c c c r r r r}
			\toprule
			\textbf{Algorithm} & $\bm{\kappa}$ & $\bm{\sigma}$ & \textbf{Gap} & \textbf{NI} & \textbf{T} & \textbf{Std(T)} & \textbf{SI} \\
			\midrule
			\multicolumn{8}{c}{$n=128$, $T_{\max}=40$~s} \\
			\midrule
			\multirow{3}{*}{CV1}
			& 0.8  & --  & $6.1\cdot 10^{-1}$ & 3\,429 & 40.0$^\dagger$ & 0.0 & -- \\
			& 0.9  & --  & $4.3\cdot 10^{-1}$ & 3\,407 & 40.0$^\dagger$ & 0.0 & -- \\
			& 0.99 & --  & $3.1\cdot 10^{-1}$ & 3\,428 & 40.0$^\dagger$ & 0.0 & -- \\
			\addlinespace
			CV-sum & 0.8 & -- & $3.0\cdot 10^{-1}$ & 2\,782 & 40.0$^\dagger$ & 0.0 & 1.01 \\
			\addlinespace
			\multirow{3}{*}{EICV}
			& 0.8  & 0.1 & $9.9\cdot 10^{-5}$ & 244 & 14.4 & 2.3 & 66.67 \\
			& 0.8  & 0.5 & $1.0\cdot 10^{-4}$ & 247 & 11.5 & 1.9 & 49.13 \\
			& 0.8  & 0.9 & $9.9\cdot 10^{-5}$ & 252 & {\bf 10.9} & 1.9 & 44.44 \\
			\midrule
			\multicolumn{8}{c}{$n=256$, $T_{\max}=300$~s} \\
			\midrule
			\multirow{3}{*}{CV1}
			& 0.8  & --  & $4.3\cdot 10^{-2}$ & 25\,943 & 300.0$^\dagger$ & 0.0 & -- \\
			& 0.9  & --  & $3.0\cdot 10^{-2}$ & 25\,958 & 300.0$^\dagger$ & 0.0 & -- \\
			& 0.99 & --  & $2.2\cdot 10^{-2}$ & 25\,920 & 300.0$^\dagger$ & 0.0 & -- \\
			\addlinespace
			CV-sum & 0.8 & -- & $6.5\cdot 10^{-2}$ & 17\,046 & 300.0$^\dagger$ & 0.0 & 1.00 \\
			\addlinespace
			\multirow{3}{*}{EICV}
			& 0.8  & 0.1 & $1.0\cdot 10^{-4}$ & 225 & 94.0 & 3.7 & 127.09 \\
			& 0.8  & 0.5 & $1.0\cdot 10^{-4}$ & 227 & 71.8 & 2.6 & 94.78 \\
			& 0.8  & 0.9 & $1.0\cdot 10^{-4}$ & 232 & {\bf 66.0} & 2.5 & 85.43 \\
			\bottomrule
		\end{tabular}
		\caption{Average over five noise realizations of the relative objective gap $(F(x_n)-F^\star)/|F^\star|$ at the stopping time (Gap), number of iterations (NI), computational time in seconds (T), its standard deviation (Std(T)), and mean value of subiterations (SI). $^\dagger$: the tolerance $10^{-4}$ was not reached within $T_{\max}$ in any realization. The best time is highlighted in {\bf black}.}
		\label{T:results}
	\end{table}
	
	\begin{figure}
		\centering
		\subfloat[$\overline{x}$]{\label{fig:phantom}\includegraphics[width=0.23\textwidth]{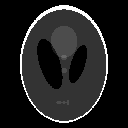}}\hfill
		\subfloat[CV1, $23.35$]{\label{fig:CV1}\includegraphics[width=0.23\textwidth]{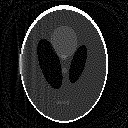}}\hfill
		\subfloat[CV-sum, $23.31$]{\label{fig:CVsum}\includegraphics[width=0.23\textwidth]{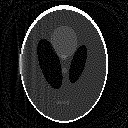}}\hfill
		\subfloat[EICV, $26.90$]{\label{fig:EICV}\includegraphics[width=0.23\textwidth]{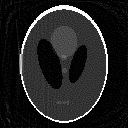}}
		\caption{Phantom test image $\overline{x}$ for $n=128$ and reconstructed images (first noise realization) obtained by the best instance of each algorithm at its stopping time, with their Peak Signal-to-Noise Ratio (PSNR).} \label{fig:rec}
	\end{figure}
	\begin{figure}
		\centering
		\includegraphics[width=\textwidth]{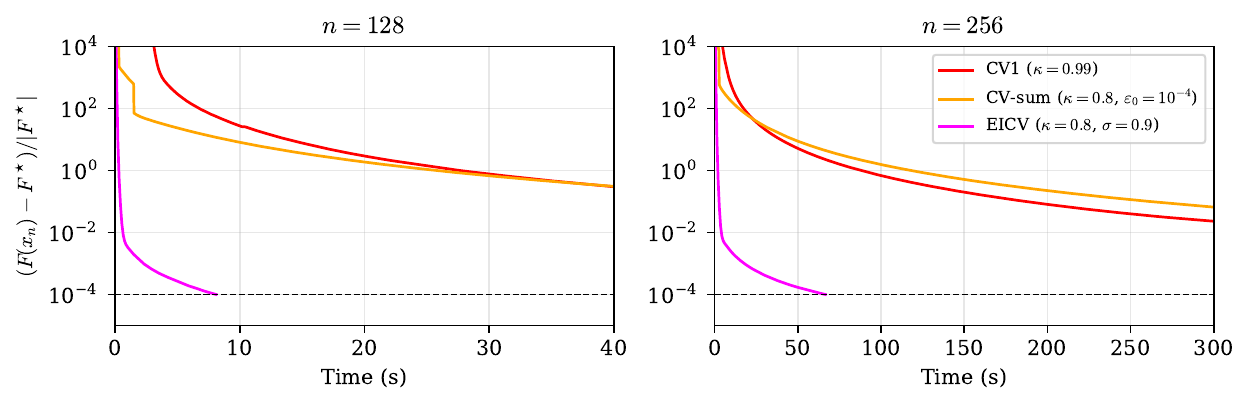}
		\caption{Relative objective gap vs time (s) for the best instance of each algorithm (first noise realization).} \label{fig:grap1}
	\end{figure}
\section{Conclusion}\label{sec:conclu} 
In this article, we have proposed an inexact warped resolvent framework for structured monotone inclusions, based on relative-error evaluations and separating halfspace constructions. This geometric viewpoint provides a unified interpretation of several splitting schemes and allows for inexact backward computations in both standard and primal--dual settings. The proposed methods are shown to converge weakly under mild assumptions, strong convergence via Haugazeau-type projection steps, and linear convergence under metric subregularity. Numerical experiments on saddle-point problems and computed tomography reconstruction illustrate the flexibility of the framework and its practical potential in significantly reducing computational costs for large-scale problems. 
%As future 
\section*{Acknowledgments}
The first author was partially supported by Centro de Modelamiento Matemático (CMM) BASAL Fund FB210005 for Centers of Excellence, and FONDECYT Iniciación Grant 11261620. The second author was partially supported by ANID through FONDECYT Iniciación Grant 11250164.
%\bibliographystyle{spmpsci.bst}
%\bibliography{ref}	

\end{document}